\documentclass[a4paper,11pt]{article}

\usepackage{float}
\usepackage{amsmath,amssymb,amsthm,amscd}
\usepackage[mathscr]{eucal}
\usepackage[all]{xy}
\usepackage{graphicx}
\usepackage{color}
\usepackage{hyperref}
\hypersetup{
colorlinks=true,
linkcolor=blue,
citecolor=blue,
urlcolor=blue
}

\makeatletter
    
    \@addtoreset{equation}{section}
  \makeatother

\newcommand{\mcal}{\mathcal}
\newcommand{\mbf}{\mathbf}

\newcommand{\mfrak}{\mathfrak}
\newcommand{\mbb}{\mathbb}
\newcommand{\mrm}{\mathrm}
\newcommand{\vphi}{\varphi}
\newcommand{\ve}{\varepsilon}
\newtheorem{theorem}{Theorem}[section]
\newtheorem{corollary}[theorem]{Corollary}
\newtheorem{lemma}[theorem]{Lemma}
\newtheorem{proposition}[theorem]{Proposition}

\theoremstyle{definition}

\newtheorem{remark}[theorem]{Remark}

\title{Fields of definition of $p$-torsion points of elliptic curves and their ramification}
\author{Yoshiyasu Ozeki\footnote{
Faculty of Science, Kanagawa University,
3-27-1 Rokkakubashi, Kanagawa-ku, Yokohama-shi, Kanagawa 221-8686, JAPAN
\endgraf
e-mail: {\tt ozeki@kanagawa-u.ac.jp}
}
and Manabu Yoshida\footnote{
Faculty of Mathematics Education, Osaka Kyoiku University, 
4-698-1 Asahigaoka, Kashiwara, Osaka 582-8582, JAPAN
\endgraf
e-mail: {\tt yoshida-m95@cc.osaka-kyoiku.ac.jp}
\endgraf
}}

\begin{document}
\maketitle

\begin{abstract}
Let $E$ be an elliptic curve over $\mbb{Q}_p$. 
We study the field $\mbb{Q}_p(E[p])$ generated by the $p$-torsion points of $E$. 
When $E$ has good reduction, we determine not only $\mbb{Q}_p(E[p])$ 
but also the field $\mbb{Q}_p(P)$ for every point $P\in E[p]$. 
This classification yields criteria for the existence of $p$-torsion over 
unramified and Lubin--Tate extensions of $\mbb{Q}_p$. 
As a global application, let $E/\mbb{Q}$ be an elliptic curve and let $p$ be an 
odd prime of good reduction. Define $f_p(E)$ to be the multiplicative order of 
$a_p(E)$ modulo $p$ in the ordinary case, and set $f_p(E)=p^2-1$ in the 
supersingular case. We prove that $E(K)[p]=0$ for every number field 
$K/\mbb{Q}$ with $[K:\mbb{Q}]<f_p(E)$. 
When $E$ has bad reduction, we describe the ramified part of 
$\mbb{Q}_p(E[p])/\mbb{Q}_p$. For every reduction type, we determine the 
maximal upper ramification break of $\mbb{Q}_p(E[p])/\mbb{Q}_p$.
\end{abstract}

    \tableofcontents


\section{Introduction}

Let $p$ be a prime and let $\mbb{Q}_p$ be the field of $p$-adic numbers. 
For an elliptic curve $E/\mbb{Q}_p$, we study the field 
$\mbb{Q}_p(E[p])$ generated by its $p$-torsion points. 
The shape of this field, and the methods used to analyze it, depend strongly on 
the reduction type of $E$.

In the good-reduction case, classical results of Serre determine the possible 
degrees of the full $p$-torsion field. If $E$ has good ordinary reduction, then
\[
 [\mbb{Q}_p(E[p]):\mbb{Q}_p]=(p-1)f
 \quad\text{or}\quad
 p(p-1)f,
\]
where $f$ is the residue degree of $\mbb{Q}_p(E[p])/\mbb{Q}_p$ 
\cite[Proposition 11, Corollaire]{Ser72}. If $E$ has good supersingular 
reduction, then 
$[\mbb{Q}_p(E[p]):\mbb{Q}_p]=2(p^2-1)$ 
\cite[Proposition 12]{Ser72}.

Our first goal is to refine these degree formulas by determining the fields of 
definition themselves. In the ordinary case, we give a criterion for whether 
$E[p]$ is tamely or wildly ramified (Lemma~\ref{lem:tameness}) and determine 
$\mbb{Q}_p(P)$ for every $P\in E[p]$ 
(Theorems~\ref{ord:tame}, \ref{ord:wild}, and \ref{ord:p=2}). 
In the supersingular case, we likewise determine the field of definition of 
every $p$-torsion point (Theorem~\ref{ss}). Thus, in the good-reduction case, 
our results describe not only the full field $\mbb{Q}_p(E[p])$ but also how its 
subfields arise from individual torsion points.

For bad reduction, we focus on the ramified part of 
$\mbb{Q}_p(E[p])/\mbb{Q}_p$. Kraus computed the valuation of the different of
\[
 \mbb{Q}_p^{\mrm{ur}}(E[p])/\mbb{Q}_p^{\mrm{ur}},
\]
where $\mbb{Q}_p^{\mrm{ur}}$ is the maximal unramified extension in a fixed 
algebraic closure of $\mbb{Q}_p$ \cite{Kra99}. Combining these computations 
with the classification of degree-$p$ extensions, we determine the structure 
of the ramified part and the maximal upper ramification break for every bad 
reduction type 
(Theorems~\ref{thm:PotentialMulti}, \ref{thm:Multiplicative:p=2}, 
\ref{thm:PotentialGood}, \ref{thm:Additive:p=3}, and 
\ref{thm:Additive:p=2}).

The explicit description of the fields $\mbb{Q}_p(P)$ also has applications to 
the existence of $p$-torsion over local and global extensions. 
Corollary~\ref{cor:unramified-torsion} characterizes the appearance of 
$p$-torsion over finite unramified extensions, while 
Theorem~\ref{thm:lubin-tate-torsion} gives a congruence criterion for 
$p$-torsion over Lubin--Tate extensions. The main global application is 
Theorem~\ref{thm:global}. For an elliptic curve $E/\mbb{Q}$ and an odd prime 
$p$ of good reduction, define
\[
 f_p(E)=
 \begin{cases}
  \mrm{ord}_{\mbb{F}_p^{\times}}\bigl(a_p(E)\bigr),
    & \text{if $E$ has ordinary reduction at $p$},\\
  p^2-1,
    & \text{if $E$ has supersingular reduction at $p$}.
 \end{cases}
\]
Then, for every number field $K/\mbb{Q}$ of degree $d$, 
Theorem~\ref{thm:global} gives the explicit vanishing criterion
\[
 f_p(E)>d \quad\Longrightarrow\quad E(K)[p]=0.
\]
In particular, a single local invariant at $p$ provides an obstruction to the 
existence of $p$-torsion over number fields of prescribed degree.

We also mention the related work of Freitas and Kraus \cite{FrKr20}, who 
determined the degree of $\mbb{Q}_{\ell}(E[p])/\mbb{Q}_{\ell}$ for 
$\ell\neq p$ and $p>2$. Our study treats the complementary case $\ell=p$ and 
provides a detailed description of the ramification of the $p$-torsion field.

The paper is organized as follows. In Section~\ref{main:local}, we state the 
main results. We first determine the fields generated by $p$-torsion points in 
the good-reduction case, then give the local and global applications in 
Subsection~\ref{subsec:applications}, and finally determine the ramification 
breaks for all reduction types; explicit examples are collected at the end of 
the section. In Section~\ref{sect:Ramification}, we review the ramification 
theory used in the proofs and analyze the relevant extensions of degree $p$. 
In Section~\ref{Sect:Proof}, we prove the main results. For good reduction, we 
use the ramification theory of finite flat group schemes \cite{Fon85}. This 
approach differs from that of Freitas and Kraus \cite{FrKr20}, who use the 
Serre--Tate theorem \cite[Lemma 2]{SeTa68} for primes $\ell\neq p$. For bad 
reduction, we combine Kraus' computations of the different with 
Propositions~\ref{prop:breaklambda} and \ref{prop:break0} to recover defining 
polynomials for the ramified extension $L/F$, where 
$L=\mbb{Q}_p(E[p])$ and $F$ is its maximal unramified subextension. We retain 
Kraus' division into cases throughout.

\vspace{5mm}
\noindent
{\bf Notation.}
Let $p$ be a prime.
In this paper, we fix algebraic closures 
$\overline{\mbb{Q}}_p$ and $\overline{\mbb{F}}_p$ 
of $\mbb{Q}_p$ and $\mbb{F}_p$, respectively. 
Let $v_p$ be the valuation on $\overline{\mbb{Q}}_p$ normalized by 
$v_p(\mbb{Q}_p^{\times})=\mbb{Z}$. 
We denote by $\mbb{Q}_{p^f}$ the unramified extension of degree $f$ over $\mbb{Q}_p$. 
For a finite extension $K$ of $\mbb{Q}_p$, 
we  denote by $\mcal{O}_K$ and $\mbf{m}_K$ 
the ring of integers of $K$ and its maximal ideal, respectively. 
Let $v_K$ be the valuation on $K$ normalized by $v_K(K^{\times})=\mbb{Z}$. 
Denote by $\mfrak{R}_K$ a complete system of representatives of $\mcal{O}_K$ 
modulo $\mbf{m}_K$ such that $0 \in \mfrak{R}_K$. 
We set $\Gamma_K:=\mrm{Hom}_{\mbb{Q}_p}(K,\overline{\mbb{Q}}_p)$
and denote by $G_K$  the absolute Galois group 
$\mrm{Gal}(\overline{\mbb{Q}}_p/K)$ of $K$.
We denote by $\mu_p$ the set of $p$-th roots of unity in $\overline{\mbb{Q}}_p$. 
The labels of elliptic curves in this paper follow 
the convention used in the Cremona database. 
The data available in the LMFDB \cite{lmfdb} is also useful for referencing elliptic curves; 
however, note that the labeling in the LMFDB differs from 
that of the Cremona label. 

\section{Main results}
\label{main:local}
Let $E$ be an elliptic curve over $\mbb{Q}_p$.
In this section, we summarize our results.
First, we determine  the explicit structure of the field $\mbb{Q}_p(E[p])$ 
generated by all $p$-torsion points $E[p]$ of $E$,   
as well as the field $\mbb{Q}_p(P)$ generated by a given point $P\in E[p]$, 
in the case where $E$ has good reduction.
Next, combining our results with those  of \cite{Kra99},
we determine the maximal ramification breaks of the extension 
$\mbb{Q}_p(E[p])/\mbb{Q}_p$ without any reduction hypothesis on $E$.

\subsection{On the structure of the field generated by a {\it p}-torsion point}
If $E$ has good reduction, we denote by $\bar{E}$ the reduction of $E$. 
We set $a_p(E):=1+p-\# \bar{E}(\mbb{F}_p)$.
The Hasse bound gives $|a_p(E)|\le 2\sqrt{p}$.
\subsubsection{Case of good ordinary reduction}
In this section,
we assume that $E$ has {\it good ordinary reduction}, that is, 
$\bar{E}(\overline{\mbb{F}}_p)[p]\not=0$. 
This is equivalent to saying that 
$a_p(E)\not \equiv 0 \! \mod p$.
Let $f$ be the order of $a_p(E) \! \mod p$
and denote by $F$ the unramified extension of $\mbb{Q}_p$ of degree $f$.
Then, it is known that the following are equivalent for an odd prime $p$ 
(see Lemma \ref{lem:tameness} below).
\begin{itemize}
\item[(i)] $E(F)[p]\not=0$.
\item[(ii)] $E[p]\simeq \mbb{F}_p(1)\oplus \mbb{F}_p$ as $G_F$-modules.
\item[(iii)] $E[p]$ is tame.
\end{itemize}
We remark that, since $F$ is unramified over $\mbb{Q}_p$,
the condition $E(F)[p]\not=0$ is equivalent to saying that $E(F)[p]\simeq \mbb{Z}/p\mbb{Z}$.
By \cite[Proposition 13.7 and \S 17]{Gro90}, the above condition is equivalent to 
$j(E)\equiv j(E_0)  \pmod{p^2}$, 
where  $E_0$ is the canonical lift of $\bar{E}$ (see \cite[\S17]{Gro90}). 
The structures of the fields $\mbb{Q}_p(E[p])$ and $\mbb{Q}_p(P)$ for $P\in E[p]$ 
vary depending on whether 
$E[p]$ is tame or wild.
The Galois group of the extension $\mbb{Q}_p(E[p])/\mbb{Q}_p$ is well-known
(cf.\ \cite[Proposition 11, Corollaire]{Ser72}).  
Put $G=\mrm{Gal}(\mbb{Q}_p(E[p])/\mbb{Q}_p)$
and denote by $I$ the inertia subgroup of $\mrm{Gal}(\mbb{Q}_p(E[p])/\mbb{Q}_p)$.
Then, under the natural injection $G\hookrightarrow GL_{\mbb{F}_p}(E[p])\simeq GL_2(\mbb{F}_p)$
for  a suitable choice of basis,  
the images of $I\subset G$ are of the form 
\[
\begin{pmatrix}
    \mbb{F}_p^{\times} & 0 \\ 0 & 1
\end{pmatrix}
\subset 
\begin{pmatrix}
    \mbb{F}_p^{\times} & 0 \\ 0 & C_f
\end{pmatrix}
\quad 
\mbox{or}
\quad 
\begin{pmatrix}
    \mbb{F}_p^{\times} & \mbb{F}_p \\ 0 & 1
\end{pmatrix}
\subset 
\begin{pmatrix}
    \mbb{F}_p^{\times} & \mbb{F}_p \\ 0 & C_f
\end{pmatrix}
\]
if $E[p]$ is tame or wild, respectively.
The results below focus on field structures of $\mbb{Q}_p(E[p])$ and  $\mbb{Q}_p(P)$ for $P\in E[p]$.

\begin{theorem}[Tame  case: $p\ge 3$]
\label{ord:tame}
Let $p$ be an odd prime and assume $E(F)[p]\not=0$.
\begin{itemize}
\item[{\rm (1)}] $\mbb{Q}_p(E[p])=F(\mu_p)$. 
In particular, $[\mbb{Q}_p(E[p]):\mbb{Q}_p]=f(p-1)$.
\item[{\rm (2)}] For each non-zero $P\in E[p]$, we have 
\begin{align*}
\mbb{Q}_p(P)=
\left\{
\begin{array}{cl}
F &
if\ P\in E(F)[p], \cr
\mbb{Q}_p\left(\sqrt[p-1]{-a_p(E)^{-1}p}\right) &
if\ P\in \mrm{ker}(E[p]\overset{\mrm{red}}{\rightarrow} \bar{E}[p]), \cr
F(\mu_p) & 
\mbox{otherwise}.
\end{array}
\right.
\end{align*}
\end{itemize}
Here, $E[p]\overset{\mrm{red}}{\rightarrow} \bar{E}[p]$ is the reduction map.
\end{theorem}
\begin{theorem}[Wild case: $p\ge 3$]
\label{ord:wild}
Let $p$ be an odd prime and assume that $E(F)[p]=0$.
\begin{itemize}
\item[{\rm (1)}] The field extension 
$\mbb{Q}_p(E[p])/\mbb{Q}_p$ has exactly $p$ subextensions
of degree $p$, all of which are conjugate over $\mbb{Q}_p$.
Moreover, they are all $\mathbb{Q}_p$-isomorphic to 
\[
\mbb{Q}_p[X]/(X^p+a_p(E)^{-2}pX+p).
\]
\end{itemize}
\noindent
In the following, let $K$ denote the degree $p$ subextension of 
$\mbb{Q}_p(E[p])/\mbb{Q}_p$ described above.
\begin{itemize}
\item[{\rm (2)}] $\mbb{Q}_p(E[p])=FK(\mu_p)$. In particular, $[\mbb{Q}_p(E[p]):\mbb{Q}_p]=fp(p-1)$.
\item[{\rm (3)}] The reduction map gives an isomorphism $E(FK)[p]\simeq \bar{E}[p]$. 
\item[{\rm (4)}] For each $\sigma\in\Gamma_K$, the group
$E(F\cdot\sigma K)[p]$ has exactly $p-1$ non-zero elements.
Furthermore, for each non-zero $P\in E[p]$, we have
\begin{align*}
\mbb{Q}_p(P)=
\left\{
\begin{array}{cl}
\mbb{Q}_p\left(\sqrt[p-1]{-a_p(E)^{-1}p}\right) &
if\ P\in \mrm{ker}(E[p]\overset{\mrm{red}}{\rightarrow} \bar{E}[p]), \cr
F\cdot \sigma K &
if\ P\in E(F\cdot\sigma K)[p],\ \sigma\in\Gamma_K.
\end{array}
\right.
\end{align*}
Here, $E[p]\overset{\mrm{red}}{\rightarrow} \bar{E}[p]$ is the reduction map.
\end{itemize}
\end{theorem}

We note that the fields $\mbb{Q}_p\left(\sqrt[p-1]{-a_p(E)^{-1}p}\right)$  and $F\cdot \sigma K$ with $\sigma\in \Gamma_K$
appearing in (4) of the above theorem
are all distinct from each other. 
Thus the assertion (4) of the  theorem investigates the field of definition for 
every non-zero point in $E[p]$. 

Combining Theorem \ref{ord:tame} and \ref{ord:wild}, 
we obtain the following criterion for $p$-torsion growth: 
\begin{corollary}\label{cor:torsion-growth}
Let $p$ be an odd prime, let $E$ be an elliptic curve over $\mbb{Q}_p$ 
with good ordinary reduction, and assume that $E(\mbb{Q}_p)[p]=0$.  
Let $f$ be the order of $a_p(E) \pmod{p}$, and let $F/\mbb{Q}_p$ be 
the unramified extension of degree $f$. 
Set $M=\mbb{Q}_p\left(\sqrt[p-1]{-a_p(E)^{-1}p}\right)$. 
Then, for any finite extension $L/\mbb{Q}_p$, we have 
$E(L)[p] \neq 0$ if and only if either $M \subset L$, or one of the following 
holds:
\begin{itemize}
\item[{\rm (i)}] $E(F)[p] \neq 0$ and $F \subset L$; 
\item[{\rm (ii)}] $E(F)[p]=0$ and there exists an extension 
$K/\mbb{Q}_p$ such that $FK\subset L$ and
\[
K\simeq \mbb{Q}_p[X]/(X^p+a_p(E)^{-2}pX+p)
\]
as $\mbb{Q}_p$-algebras.
\end{itemize}
\end{corollary}
\begin{proof}
If $E(F)[p] \neq 0$, the result follows from Theorem \ref{ord:tame}. 
If $E(F)[p]=0$, the result follows from Theorem \ref{ord:wild}. 
\end{proof}
\begin{theorem}[$p=2$]
\label{ord:p=2}
 $\mbb{Q}_2(E[2])$ is $\mbb{Q}_2$-isomorphic to 
one of the following fields.
\begin{itemize}
\item[{\rm (i)}] $\mbb{Q}_2$, 
\item[{\rm (ii)}] the unramified quadratic extension $\mbb{Q}_4$ of $\mbb{Q}_2$, 
\item[{\rm (iii)}] $L_1:=\mbb{Q}_2[X]/(X^2+2X+2)$,  
\item[{\rm (iv)}] $L_2:=\mbb{Q}_2[X]/(X^2+2X+6)$.
\end{itemize}
If $\mbb{Q}_2(E[2])\in \{\mbb{Q}_4, L_1, L_2 \}$,
then, $\mbb{Q}_2(P)=\mbb{Q}_2$ for exactly two elements $P$ of $E[2]$
and $\mbb{Q}_2(P)=\mbb{Q}_2(E[2])$ for the other two elements $P$ of $E[2]$.
\end{theorem}
The elliptic curve 
$E=$\href{https://www.lmfdb.org/EllipticCurve/Q/17a3/}{17a3} 
corresponds to (i), 
$E=$\href{https://www.lmfdb.org/EllipticCurve/Q/15a5/}{15a5} 
to (ii), 
$E=$\href{https://www.lmfdb.org/EllipticCurve/Q/49a4/}{49a4} 
to (iii), 
and 
$E=$\href{https://www.lmfdb.org/EllipticCurve/Q/15a6/}{15a6} 
to (iv). 
Thus, all four cases in Theorem \ref{ord:p=2} can actually occur.

\subsubsection{Case of good supersingular reduction}

In this section,
we assume that $E$ has {\it good supersingular reduction}, that is, 
$\bar{E}(\overline{\mbb{F}}_p)[p]=0$. 
This is equivalent to saying that 
$a_p(E)\equiv 0 \! \mod p$.
By the Hasse bound, we have
$a_p(E)=0$ for $p\ge 5$.

The Galois group of the extension $\mbb{Q}_p(E[p])/\mbb{Q}_p$ is studied in \cite[Proposition 12]{Ser72}.  
Put $G=\mrm{Gal}(\mbb{Q}_p(E[p])/\mbb{Q}_p)$
and denote by $I$ the inertia subgroup of $\mrm{Gal}(\mbb{Q}_p(E[p])/\mbb{Q}_p)$.
Then, under the natural injection $G\hookrightarrow GL_{\mbb{F}_p}(E[p])$,  
the image of $I$ is a non-split Cartan subgroup which is isomorphic to $\mbb{F}_{p^2}^{\times}$
and $G$ is the normalizer of $I$ in $GL_{\mbb{F}_p}(E[p])$ (and thus $[G:I]=2$).

\begin{theorem}
\label{ss}
Let $F$ be the unramified quadratic  extension of $\mbb{Q}_p$.
\begin{itemize}
    \item[{\rm (1)}] $\mbb{Q}_p(E[p])=F(\sqrt[p^2-1]{p})$. 
    \item[{\rm (2)}] The field extension 
$\mbb{Q}_p(E[p])/\mbb{Q}_p$ has exactly $p+1$ subextensions
which are $\mbb{Q}_p$-isomorphic to  
\[
\mbb{Q}_p[X]/(X^{p^2-1}-p).
\]
For each such subextension $K$, 
there exist exactly $p-1$ non-zero points $P$ in $E[p]$
such that $\mbb{Q}_p(P)=K$.
\end{itemize}
\end{theorem}
Note that the assertion (2) of the above theorem investigates the field of definition for 
every non-zero point in $E[p]$. 
In particular, it shows that every non-zero point in 
$E[p]$ generates a totally ramified extension of degree $p^2-1$.
\begin{corollary}\label{cor:torsion-ss}
Let $E$ be an elliptic curve with good supersingular reduction over $\mbb{Q}_p$.
Then, for any finite extension $L/\mbb{Q}_p$, we have
$E(L)[p]\neq 0$ if and only if $L$ contains a subfield
$K/\mbb{Q}_p$ such that
\[
K\simeq \mbb{Q}_p[X]/(X^{p^2-1}-p).
\]
\end{corollary}
\subsection{Applications to \texorpdfstring{$p$}{p}-torsion over local and global fields}\label{subsec:applications}
In this subsection, we apply the classification of the fields of definition 
of $p$-torsion points over local and global fields.

\subsubsection{Local fields}
We first give criteria for the existence of $p$-torsion over 
unramified and Lubin-Tate extensions of $\mbb{Q}_p$.
\begin{corollary}\label{cor:unramified-torsion}
Let $E$ be an elliptic curve over $\mbb{Q}_p$ with good reduction, 
and let $L/\mbb{Q}_p$ be a finite unramified extension.
Suppose that $p>2$ and $E$ has ordinary reduction. 
Let $f$ be the order of $a_p(E) \pmod{p}$, and let $F/\mbb{Q}_p$ 
be the unramified extension of degree $f$. 
Then $E(L)[p] \neq 0$ if and only if $E(F)[p] \neq 0$ and $f \mid [L:\mbb{Q}_p]$.
\end{corollary}
\begin{proof}
The assertion follows from Theorem \ref{ord:tame} 
and \ref{ord:wild}. 
\end{proof}

\begin{remark}
When $p=2$ and $E$ has ordinary reduction, 
we have $E(\mbb{Q}_2)[2] \neq 0$ by Theorem \ref{ord:p=2}. 
The possible torsion subgroups $E(\mbb{Q}_2)_{\mrm{tor}}$ were determined 
in \cite[Theorem 1.1(2)]{OzYo25}. 
\end{remark}
\begin{theorem}\label{thm:lubin-tate-torsion}
Let $p$ be an odd prime, let $n\geq 1$, and let $E$ be an elliptic curve over $\mbb{Q}_p$ 
with good ordinary reduction. Assume that $E(\mbb{Q}_p)[p]=0$. 
For $u \in \mbb{Z}_p^{\times}$, let $L_{\pi,n}/\mbb{Q}_p$ denote 
the $n$-th Lubin-Tate extension 
associated with the uniformizer $\pi=up$. 
Then $E(L_{\pi,n})[p] \neq 0$ if and only if $u \equiv a_p(E)^{-1} \pmod{p}$. 
\end{theorem}
\begin{proof}
Put
$M=\mbb{Q}_p\left(\sqrt[p-1]{-a_p(E)^{-1}p}\right)$.
Let $f$ be the order of $a_p(E)\pmod p$, and let
$F/\mbb{Q}_p$ be the unramified extension of degree $f$, as in
Corollary~\ref{cor:torsion-growth}. The extension
$L_{\pi,n}/\mbb{Q}_p$ is totally ramified and abelian. If $f>1$,
neither alternative (i) nor alternative (ii) of that corollary can
occur, since either one would force $F\subset L_{\pi,n}$. If $f=1$,
then $F=\mbb{Q}_p$; alternative (i) contradicts
$E(\mbb{Q}_p)[p]=0$, while in alternative (ii) the extension
$K/\mbb{Q}_p$ is non-Galois by
Proposition~\ref{prop:GaloisCriteria}, and hence cannot be contained
in the abelian extension $L_{\pi,n}$. Therefore,
$E(L_{\pi,n})[p]\neq 0$ if and only if $M\subset L_{\pi,n}$.
The extension $M/\mbb{Q}_p$ is tamely ramified of degree $p-1$,
and $L_{\pi,1}$ is the unique tamely ramified subextension of
$L_{\pi,n}/\mbb{Q}_p$ of degree $p-1$. Thus the last inclusion is
equivalent to $L_{\pi,1}=M$.

Note that
\[
L_{\pi,1}=\mbb{Q}_p\left(\sqrt[p-1]{-up}\right).
\]
For $u,u'\in\mbb{Z}_p^{\times}$, if $u\equiv u'\pmod p$, then
\[
\mbb{Q}_p\left(\sqrt[p-1]{-up}\right)
=\mbb{Q}_p\left(\sqrt[p-1]{-u'p}\right).
\]
This also follows from Proposition~\ref{prop:criteria}. Since there
are exactly $p-1$ totally ramified abelian extensions of
$\mbb{Q}_p$ of degree $p-1$, there is a bijection between the
residue classes $u\pmod p$ and these extensions. Thus the result
follows.
\end{proof}
\begin{remark}
When $u=1$, the Lubin--Tate extension $L_{p,n}/\mbb{Q}_p$ coincides with
the cyclotomic extension $\mbb{Q}_p(\mu_{p^n})/\mbb{Q}_p$.
In this case, \cite[Proposition 3.3]{OzYo25} gives
\[
E(\mbb{Q}_p(\mu_{p^n}))[p]=0
\quad\Longleftrightarrow\quad
\bar{E}(\mbb{F}_p)[p]=0.
\]
\end{remark}
\subsubsection{Number fields}
For a prime number $p$ and an elliptic curve $E/\mbb{Q}$, 
we define 
\[ f_p(E)=
\begin{cases}
    \mrm{ord}_{\mbb{F}_p^{\times}} \left(a_p(E)\right) & 
    \text{if $E$ has ordinary reduction at $p$}, \\
    p^2-1 & \text{if $E$ has supersingular reduction at $p$}.
\end{cases}
\]
The following theorem gives a local criterion for the existence 
of $p$-torsion over number fields.
\begin{theorem}\label{thm:global}
Let $E/\mbb{Q}$ be an elliptic curve and $K/\mbb{Q}$ 
a number field of degree $d$. 
Let $p$ be an odd prime of good reduction for $E$. 
If $f_p(E)>d$, then $E(K)[p]=0$. 
\end{theorem}
\begin{proof}
Let $w$ be a prime of $K$ above $p$, and let $K_w$ denote the completion 
of $K$ at $w$. 
Suppose that $E(K)[p] \neq 0$. 
Then $E(K_w)[p] \neq 0$. 
Put $d_w=[K_w:\mbb{Q}_p]$. 
By Theorem \ref{ord:tame} and \ref{ord:wild}, 
together with Corollary \ref{cor:torsion-ss}, 
we have $d_w \geq f_p(E)$. 
Hence $d \geq d_w \geq f_p(E)$, which contradicts the assumption $f_p(E)>d$. 
\end{proof}
\begin{remark}
For a fixed elliptic curve $E/\mbb{Q}$ and a fixed number field $K$,
the Mordell--Weil theorem implies that $E(K)_{\mrm{tor}}$ is finite.
In particular, $E(K)[p]=0$ for all sufficiently large primes $p$.
Theorem~\ref{thm:global} provides an explicit sufficient condition
for this vanishing at a given odd prime $p$ of good reduction,
namely $f_p(E)>[K:\mbb{Q}]$.
Here $f_p(E)$ is the multiplicative order of $a_p(E)$ modulo $p$
in the ordinary case, and $f_p(E)=p^2-1$ in the supersingular case.
\end{remark}
\subsection{Ramification breaks of the field of \texorpdfstring{$p$}{p}-torsion points}
As a consequence, we obtain the ramification break of 
$\mbb{Q}_p(E[p])/\mbb{Q}_p$. 
For a finite extension $L$ of $K$, 
we denote the different of $L/K$ by $\mathfrak{D}_{L/K}$ 
and the upper numbering ramification break of $L/K$ 
by $u_{L/K}$ (see Section \ref{sect:Ramification} for details). 

\subsubsection{Case of good reduction}
\begin{theorem}\label{cor:defferent}
    Let $E$ be an elliptic curve over $\mbb{Q}_p$ with good reduction.
    Put $L=\mbb{Q}_p(E[p])$.
    If $\bar{E}$ is ordinary, let $f$ be the order of
    $a_p(E)\pmod p$ and let $F$ be the unramified extension of
    $\mbb{Q}_p$ of degree $f$.
\begin{itemize}    
    \item[{\rm (1)}]
    Suppose $\bar{E}$ is ordinary and $p \geq 3$.
    Then we have:
    \begin{itemize}
        \item[\rm (i)]
        $v_p(\mfrak{D}_{L/\mbb{Q}_p})=\frac{p-2}{p-1}$ and 
        $u_{L/\mbb{Q}_p}=1$ if $E(F)[p] \not=0$,
        \item[\rm (ii)] 
        $v_p(\mfrak{D}_{L/\mbb{Q}_p})=\frac{p^2-2}{p^2-p}$ and 
        $u_{L/\mbb{Q}_p}=\frac{p}{p-1}$ if $E(F)[p]=0$, 
    \end{itemize}

    \item[{\rm (2)}] 
    Suppose $\bar{E}$ is ordinary and $p=2$. 
    Then we have:
    \begin{itemize}
    \item[\rm (i)] 
    $v_p(\mfrak{D}_{L/\mbb{Q}_2})=u_{L/\mbb{Q}_2}=0$ if 
    $L=\mbb{Q}_2$ or $L=\mbb{Q}_4$,
    \item[\rm (ii)]
    $v_p(\mfrak{D}_{L/\mbb{Q}_2})=1$ and $u_{L/\mbb{Q}_2}=2$ if $L/\mbb{Q}_2$ 
    is ramified.
    \end{itemize}

    \item[{\rm (3)}] 
    Suppose $\bar{E}$ is supersingular.
    Then we have 
    $v_p(\mfrak{D}_{L/\mbb{Q}_p})=\frac{p^2-2}{p^2-1}$ and 
    $u_{L/\mbb{Q}_p}=1$.
\end{itemize}
\end{theorem}
\begin{remark}
These results on the different coincide with those in 
\cite[Th\'eor\`em 2]{Kra99} and 
\cite[Theorem 3]{MoTa03}. 
\end{remark}
\subsubsection{Case of potentially multiplicative reduction}
Let $E$ be an elliptic curve over $\mbb{Q}_p$. 
We write $c_4$, $c_6$ and $\Delta$ for the standard invariants of a 
minimal Weierstrass model of $E/\mbb{Q}_p$. 
Let $j=j(E)$ be the modular invariant of $E$.
In this section, we assume that $E$ satisfies $v_p(j)<0$. 
The field of definition $L:=\mbb{Q}_p(E[p])$ has been studied 
in \cite[Section 19.2]{FrKr20} and \cite[Th\'eor\`em 1]{Kra99}. 
We translate those results into our terminology and determine 
the ramification breaks $u_{L/\mbb{Q}_p}$. 
\begin{theorem}[$p \geq 3$]\label{thm:PotentialMulti}
Suppose $v_p(j)<0$ and $p$ is odd. 
Let $j=p^{v_p(j)} \tilde{j}.$

\begin{itemize}
\item[{\rm (1)}] Suppose one of the following conditions is satisfied:
\begin{itemize}
    \item[\rm (i)] $( \frac{-c_6}{p} )=1$.
    \item[\rm (ii)] $p$ divides $c_6$ and $\frac{c_6}{p}$ is a square in 
    $\mbb{Q}_p$.
\end{itemize}
Then:
\begin{itemize}
\item[{\rm (1.1)}] 
If $v_p(j) \equiv 0 \pmod{p}$ and 
$\tilde{j}^{p-1} \equiv 1 \pmod{p^2}$, then 
$L=\mbb{Q}_p(\mu_p)$ and $u_{L/\mbb{Q}_p}=1$.

\item[{\rm (1.2)}] 
If $v_p(j) \not\equiv 0 \pmod{p}$, then 
$L=K(\mu_p)$ with 
$K \simeq \mbb{Q}_p[X]/(X^p+p(1+ap))$ for some
$a \in \{0,\dots,p-1\}$ 
and $u_{L/\mbb{Q}_p}=\frac{2p-1}{p-1}$.

\item[{\rm (1.3)}] 
If $v_p(j) \equiv 0 \pmod{p}$ and 
$\tilde{j}^{p-1} \not\equiv 1 \pmod{p^2}$, then 
$L=K(\mu_p)$ with $K \simeq \mbb{Q}_p[X]/(X^p+apX+p)$ 
for some $a \in \{1,2,\dots,p-1\}$ and $u_{L/\mbb{Q}_p}=\frac{p}{p-1}$.
\end{itemize}
\end{itemize}
\begin{itemize}
\item[\rm (2)]
Suppose one of the following conditions is satisfied:
\begin{itemize}
    \item[\rm (i)] $( \frac{-c_6}{p} )=-1$.
    \item[\rm (ii)] $p$ divides $c_6$ and $\frac{c_6}{p}$ is not 
    a square in $\mbb{Q}_p$.
\end{itemize}
Then:
\begin{itemize}
\item[\rm (2.1)]
If $v_p(j) \equiv 0 \pmod{p}$ and 
$\tilde{j}^{p-1} \equiv 1 \pmod{p^2}$, then 
$L=\mbb{Q}_{p^2}(\mu_p)$ and $u_{L/\mbb{Q}_p}=1$.
\item[\rm (2.2)]
If $v_p(j) \not\equiv 0 \pmod{p}$, then $L=K(\mu_p)$ 
such that $K \simeq \mbb{Q}_{p^2}[X]/(X^p+p(1+ap))$ for some 
$a \in \mfrak{R}_{\mbb{Q}_{p^2}}$ and $u_{L/\mbb{Q}_p}=\frac{2p-1}{p-1}$.
\item[\rm (2.3)]
If $v_p(j) \equiv 0 \pmod{p}$ and 
$\tilde{j}^{p-1} \not\equiv 1 \pmod{p^2}$, 
then $L=K(\mu_p)$ such that 
$K \simeq \mbb{Q}_{p^2}[X]/(X^p+apX+p)$ for some 
$a \in \mfrak{R}_{\mbb{Q}_{p^2}} \setminus \{0\}$ 
and $u_{L/\mbb{Q}_p}=\frac{p}{p-1}$.
\end{itemize}
\end{itemize}
\end{theorem}
\begin{theorem}[$p=2$]\label{thm:Multiplicative:p=2}
Suppose $v_2(j)<0$. 
Then $L$ is one of the following fields:
\begin{itemize}
\item[\rm (i)]
$L=\mbb{Q}_2$,
\item[\rm (ii)]
$L=\mbb{Q}_4$ and $u_{L/\mbb{Q}_2}=0$,
\item[\rm (iii)]
$L=L_i$ for some $i \in \{1,2,\dots,6\}$ in Table 
\ref{table:degree2} 
and $u_{L/\mbb{Q}_2} \in \{2, 3\}$.
\end{itemize}
\end{theorem}
\subsubsection{Case of potentially good reduction}
We use the same notation as in the previous section.  
In this section, we assume that $E$ {\it has additive reduction 
and} $v_p(j) \geq 0$. 
Put $L=\mbb{Q}_p(E[p])$.
Let $F$ be the maximal unramified subextension of
$L/\mbb{Q}_p$. Kraus' results determine the ramified extension
$L/F$, but do not determine the residue degree of
$L/\mbb{Q}_p$; see \cite[Th\'eor\`em 3,4 and 5]{Kra99}. 
For $p\geq 5$, write
\[ \frac{v_p(\Delta)}{12}=\frac{a}{e} \]
with $\gcd(a,e)=1$, choose an element $u\in\overline{\mbb{Q}}_p$
such that $u^e=p^a$, and put $T=\mbb{Q}_p(u)$.
After the change of variables
\[ x=u^2X,\qquad y=u^3Y, \]
let $E_T/T$ denote the resulting elliptic curve, which has good
reduction over $T$, and let $\bar{E}_T$ denote its reduction.
Let $h$ be the height of $\bar{E}_T$. If $h=1$, let
$j_{\mathrm{can}}(\bar{E}_T)$ denote the $j$-invariant of the
canonical lift of $\bar{E}_T$.
As remarked in \cite[Remarque 3]{Kra99}, 
if $v_p(\Delta) \geq 8$, then there exists an elliptic curve $E'$ over $\mbb{Q}_p$ 
such that $c_4(E')=\frac{c_4}{p^2}$, $c_6(E')=\frac{c_6}{p^3}$, $\Delta(E')=\frac{\Delta}{p^6}$ 
and $\mbb{Q}_p^{\mrm{ur}}(E'[p])=\mbb{Q}_p^{\mrm{ur}}(E[p])$. 
Kraus \cite[Th\'eor\`em 3]{Kra99} determined the valuation of 
the different 
$v_L(\mfrak{D}_{\mbb{Q}_p^{\mrm{ur}}(E[p])/\mbb{Q}_p^{\mrm{ur}}})
=v_L(\mfrak{D}_{L/\mbb{Q}_p})$. 
We can compute the structure of $L/F$ and the ramification break $u_{L/\mbb{Q}_p}$ 
from the valuation $v_L(\mfrak{D}_{L/\mbb{Q}_p})$. 
In the following theorem, we follow the division established by Kraus. 
\begin{theorem}[$p \geq 5$]\label{thm:PotentialGood}
Suppose $E$ has additive reduction, $v_p(j) \geq 0$ and 
$p \geq 5$. 
\begin{itemize}
\item[\rm (1)] case $v_p(\Delta)=2$

\begin{itemize}
    \item[\rm (1.1)] 
    If $p \equiv 1 \pmod{3}$ and $v_p(c_4)=1$, then 
    $L=K(\mu_p)$ with $K \simeq F[X]/(X^p+apX^{(2p+1)/3}+p)$ 
    for some $a \in \mfrak{R}_F \setminus \{0\}$ and 
    $u_{L/\mbb{Q}_p}=\frac{5p-2}{3p-3}$.
    \item[\rm (1.2)]
    If $p \equiv 1 \pmod{3}$ and $v_p(c_4)\not=1$, then 
    $L=F(\mu_p)$ and $u_{L/\mbb{Q}_p}=1$.
    \item[\rm (1.3)]
    If $p \equiv 2 \pmod{3}$ and $v_p(c_4)=1$, then 
    $L=K(\mu_p)$ with $K \simeq F[X]/(X^p+apX^{(2p-1)/3}+p)$ 
    for some $a \in \mfrak{R}_F \setminus \{0\}$ 
    and $u_{L/\mbb{Q}_p}=\frac{5p-4}{3p-3}$. 
    \item[\rm (1.4)]
    If $p \equiv 5 \pmod{9}$ and $v_p(c_4)\not=1$, then 
    $L\simeq F[X]/(X^{(p^2-1)/3}+\pi_F)$ for some prime element $\pi_F$ of $F$, and $u_{L/\mbb{Q}_p}=1$.
    \item[\rm (1.5)]
    If $p \equiv 2 \pmod{3}$, $p \not\equiv 5 \pmod{9}$ 
    and $v_p(c_4) \neq 1$, then 
    $L\simeq F[X]/(X^{p^2-1}+\pi_F)$ for some prime element $\pi_F$ of $F$, and $u_{L/\mbb{Q}_p}=1$.
\end{itemize}

\item[\rm (2)] case $v_p(\Delta)=3$

\begin{itemize}
    \item[\rm (2.1)]
    If $p \equiv 1 \pmod{4}$ and $v_p(c_6)=2$, then 
    $L = K(\mu_p)$ such that $K \simeq F[X]/(X^p+apX^{(p+1)/2}+p)$ 
    for some $a \in \mfrak{R}_F \setminus \{0\}$ 
    and $u_{L/\mbb{Q}_p}=\frac{3p-1}{2p-2}$.
    \item[\rm (2.2)]
    If $p \equiv 1 \pmod{4}$ and $v_p(c_6)\not=2$, then 
    $L=F(\mu_p)$ and $u_{L/\mbb{Q}_p}=1$.
    \item[\rm (2.3)]
    If $p \equiv 3 \pmod{4}$ and $v_p(c_6)=2$, then 
    $L = K(\mu_p)$ such that $K \simeq F[X]/(X^p+apX^{(p-1)/2}+p)$ 
    for some $a \in \mfrak{R}_F \setminus \{0\}$ 
    and $u_{L/\mbb{Q}_p}=\frac{3}{2}$.
    \item[\rm (2.4)]
    If $p \equiv 3 \pmod{4}$ and $v_p(c_6)\not=2$, then 
    $L\simeq F[X]/(X^{p^2-1}+\pi_F)$ for some prime element $\pi_F$ of $F$, and $u_{L/\mbb{Q}_p}=1$.
\end{itemize}

\item[\rm (3)] case $v_p(\Delta)=4$

\begin{itemize}
    \item[\rm (3.1)] 
    If $p \equiv 1 \pmod{3}$ and $v_p(c_4)=2$, then 
    $L=K(\mu_p)$ with $K \simeq F[X]/(X^p+apX^{(p+2)/3}+p)$ 
    for some $a \in \mfrak{R}_F \setminus \{0\}$ and 
    $u_{L/\mbb{Q}_p}=\frac{4p-1}{3p-3}$.
    \item[\rm (3.2)]
    If $p \equiv 1 \pmod{3}$ and $v_p(c_4)\not=2$, then 
    $L=F(\mu_p)$ and $u_{L/\mbb{Q}_p}=1$.
    \item[\rm (3.3)]
    If $p \equiv 2 \pmod{3}$ and $v_p(c_4)=2$, then 
    $L=K(\mu_p)$ with $K \simeq F[X]/(X^p+apX^{(p-2)/3}+p)$ 
    for some $a \in \mfrak{R}_F \setminus \{0\}$ and 
    $u_{L/\mbb{Q}_p}=\frac{4p-5}{3p-3}$.
    \item[\rm (3.4)]
    If $p \equiv 2 \pmod{9}$ and $v_p(c_4)\not=2$, then 
    $L\simeq F[X]/(X^{(p^2-1)/3}+\pi_F)$ for some prime element $\pi_F$ of $F$, and $u_{L/\mbb{Q}_p}=1$.
    \item[\rm (3.5)]
    If $p \equiv 2 \pmod{3}$, $p \not\equiv 2 \pmod{9}$ 
    and $v_p(c_4) \neq 2$, then 
    $L\simeq F[X]/(X^{p^2-1}+\pi_F)$ for some prime element $\pi_F$ of $F$, and $u_{L/\mbb{Q}_p}=1$.
\end{itemize}

\item[\rm (4)] case $v_p(\Delta)=6$

\begin{itemize}
    \item[\rm (4.1)] 
    Suppose one of the following conditions is satisfied: 
    \begin{itemize}
        \item[\rm (a)]
        $p \equiv 1 \pmod{3}$ and $v_p(c_4) \geq 4$,
        \item[\rm (b)]
        $p \equiv 1 \pmod{4}$ and $v_p(c_6) \geq 5$,
        \item[\rm (c)]
        $v_p(c_4)=2$, $v_p(c_6)=3$, $h=1$ and $v_p(j-j_{\mathrm{can}}(\bar{E}_T))\not=1$.
    \end{itemize}
    Then $L=F(\mu_p)$ and $u_{L/\mbb{Q}_p}=1$. 
    \item[\rm (4.2)]
    Suppose one of the following conditions is satisfied: 
    \begin{itemize}
        \item[\rm (a)]
        $p \equiv 1 \pmod{3}$ and $v_p(c_4)=3$,
        \item[\rm (b)]
        $p \equiv 1 \pmod{4}$ and $v_p(c_6)=4$,
        \item[\rm (c)]
        $v_p(c_4)=2$, $v_p(c_6)=3$, $h=1$ and $v_p(j-j_{\mathrm{can}}(\bar{E}_T))=1$.
    \end{itemize}
    Then $L=K(\mu_p)$ with 
    $K \simeq F[X]/(X^p+apX+p)$ for some 
    $a \in \mfrak{R}_F \setminus \{0\}$ and $u_{L/\mbb{Q}_p}=\frac{p}{p-1}$.
    \item[\rm (4.3)]
    If $h=2$, then $L\simeq F[X]/(X^{p^2-1}+\pi_F)$ for some prime element $\pi_F$ of $F$, and $u_{L/\mbb{Q}_p}=1$.
\end{itemize}

\item[\rm (5)] case $v_p(\Delta)=8$

\begin{itemize}
    \item[\rm (5.1)] 
    If $p \equiv 1 \pmod{3}$ and $v_p(c_4)=3$, then 
    $L=K(\mu_p)$ with $K \simeq F[X]/(X^p+apX^{(2p+1)/3}+p)$ 
    for some $a \in \mfrak{R}_F \setminus \{0\}$ and 
    $u_{L/\mbb{Q}_p}=\frac{5p-2}{3p-3}$.
    \item[\rm (5.2)]
    If $p \equiv 1 \pmod{3}$ and $v_p(c_4)\not=3$, then 
    $L=F(\mu_p)$ and $u_{L/\mbb{Q}_p}=1$.
    \item[\rm (5.3)]
    If $p \equiv 2 \pmod{3}$ and $v_p(c_4)=3$, then 
    $L=K(\mu_p)$ with $K \simeq F[X]/(X^p+apX^{(2p-1)/3}+p)$ 
    for $a \in \mfrak{R}_F \setminus \{0\}$ 
    and $u_{L/\mbb{Q}_p}=\frac{5p-4}{3p-3}$. 
    \item[\rm (5.4)]
    If $p \equiv 5 \pmod{9}$ and $v_p(c_4)\not=3$, then 
    $L\simeq F[X]/(X^{(p^2-1)/3}+\pi_F)$ for some prime element $\pi_F$ of $F$, and $u_{L/\mbb{Q}_p}=1$.
    \item[\rm (5.5)]
    If $p \equiv 2 \pmod{3}$, $p \not\equiv 5 \pmod{9}$ 
    and $v_p(c_4) \neq 3$, then 
    $L\simeq F[X]/(X^{p^2-1}+\pi_F)$ for some prime element $\pi_F$ of $F$, and $u_{L/\mbb{Q}_p}=1$.
\end{itemize}

\item[\rm (6)] case $v_p(\Delta)=9$

\begin{itemize}
    \item[\rm (6.1)]
    If $p \equiv 1 \pmod{4}$ and $v_p(c_6)=5$, then 
    $L = K(\mu_p)$ such that $K \simeq F[X]/(X^p+apX^{(p+1)/2}+p)$ 
    for some $a \in \mfrak{R}_F \setminus \{0\}$ 
    and $u_{L/\mbb{Q}_p}=\frac{3p-1}{2p-2}$.
    \item[\rm (6.2)]
    If $p \equiv 1 \pmod{4}$ and $v_p(c_6)\not=5$, then 
    $L=F(\mu_p)$ and $u_{L/\mbb{Q}_p}=1$.
    \item[\rm (6.3)]
    If $p \equiv 3 \pmod{4}$ and $v_p(c_6)=5$, then 
    $L = K(\mu_p)$ such that $K \simeq F[X]/(X^p+apX^{(p-1)/2}+p)$ 
    for some $a \in \mfrak{R}_F \setminus \{0\}$ 
    and $u_{L/\mbb{Q}_p}=\frac{3}{2}$.
    \item[\rm (6.4)]
    If $p \equiv 3 \pmod{4}$ and $v_p(c_6)\not=5$, then 
    $L\simeq F[X]/(X^{p^2-1}+\pi_F)$ for some prime element $\pi_F$ of $F$, and $u_{L/\mbb{Q}_p}=1$.
\end{itemize}

\item[\rm (7)] case $v_p(\Delta)=10$

\begin{itemize}
    \item[\rm (7.1)] 
    If $p \equiv 1 \pmod{3}$ and $v_p(c_4)=4$, then 
    $L=K(\mu_p)$ with $K \simeq F[X]/(X^p+apX^{(p+2)/3}+p)$ 
    for some $a \in \mfrak{R}_F \setminus \{0\}$ and 
    $u_{L/\mbb{Q}_p}=\frac{4p-1}{3p-3}$.
    \item[\rm (7.2)]
    If $p \equiv 1 \pmod{3}$ and $v_p(c_4)\not=4$, then 
    $L=F(\mu_p)$ and $u_{L/\mbb{Q}_p}=1$.
    \item[\rm (7.3)]
    If $p \equiv 2 \pmod{3}$ and $v_p(c_4)=4$, then 
    $L=K(\mu_p)$ with $K \simeq F[X]/(X^p+apX^{(p-2)/3}+p)$ 
    for some $a \in \mfrak{R}_F \setminus \{0\}$ and 
    $u_{L/\mbb{Q}_p}=\frac{4p-5}{3p-3}$.
    \item[\rm (7.4)]
    If $p \equiv 2 \pmod{9}$ and $v_p(c_4)\not=4$, then 
    $L\simeq F[X]/(X^{(p^2-1)/3}+\pi_F)$ for some prime element $\pi_F$ of $F$, and $u_{L/\mbb{Q}_p}=1$.
    \item[\rm (7.5)]
    If $p \equiv 2 \pmod{3}$, $p \not\equiv 2 \pmod{9}$ 
    and $v_p(c_4) \neq 4$, then 
    $L\simeq F[X]/(X^{p^2-1}+\pi_F)$ for some prime element $\pi_F$ of $F$, and $u_{L/\mbb{Q}_p}=1$.
\end{itemize}

\end{itemize}
\end{theorem}
\begin{theorem}[$p=3$]\label{thm:Additive:p=3}
Suppose that $E$ has additive reduction, $v_3(j)\geq 0$, and $p=3$.
Let $\Delta=3^{v_3(\Delta)}\widetilde{\Delta}$.
\begin{itemize}
    \item[{\rm (1)}]
    Suppose that $v_3(\Delta)\equiv 0\pmod{3}$.
    \begin{itemize}
        \item[{\rm (1.1)}]
        If $2v_3(c_6)\leq 4+v_3(\Delta)$ and
        $\widetilde{\Delta}^{\,2}\equiv 1\pmod{9}$, then
        $L=F(\mu_3)$ and $u_{L/\mbb{Q}_3}=1$.

        \item[{\rm (1.2)}]
        If $2v_3(c_6)\leq 4+v_3(\Delta)$ and
        $\widetilde{\Delta}^{\,2}\not\equiv 1\pmod{9}$, then
        $L=FK(\mu_3)$ with
        $K\simeq \mbb{Q}_3[X]/(X^3+3X+3)$, and
        $u_{L/\mbb{Q}_3}=\frac{3}{2}$.

        \item[{\rm (1.3)}]
        If $2v_3(c_6)>4+v_3(\Delta)$, then $L\simeq F[X]/(X^8+\pi_F)$ for some prime element $\pi_F$ of $F$, and
        $u_{L/\mbb{Q}_3}=1$.
    \end{itemize}

    \item[{\rm (2)}]
    Suppose that $v_3(\Delta)\not\equiv 0\pmod{3}$.
    \begin{itemize}
        \item[{\rm (2.1)}]
        Then $L=FK(\mu_3)$ and $u_{L/\mbb{Q}_3}=\frac{5}{2}$, where $K$ is
        isomorphic to one of the three extensions\footnote{
        For comparison with the LMFDB, note that
        $\mbb{Q}_3[X]/(X^3+12)\simeq
        \mbb{Q}_3[X]/(X^3+9X+3)$ and
        $\mbb{Q}_3[X]/(X^3+21)\simeq
        \mbb{Q}_3[X]/(X^3+18X+3)$.
        }
        $\mbb{Q}_3[X]/(X^3+3)$,
        $\mbb{Q}_3[X]/(X^3+12)$, and
        $\mbb{Q}_3[X]/(X^3+21)$.
    \end{itemize}
\end{itemize}
\end{theorem}

\begin{theorem}[$p=2$]\label{thm:Additive:p=2}
Suppose that $E$ has additive reduction, $v_2(j)\geq 0$, and $p=2$.
Let $\Delta=2^{v_2(\Delta)}\widetilde{\Delta}$.
\begin{itemize}
    \item[{\rm (1)}]
    Suppose that $v_2(\Delta)$ is odd.
    \begin{itemize}
        \item[{\rm (1.1)}]
        Then $L=FL_i$ for some $i\in\{3,4,5,6\}$ in
        Table~\ref{table:degree2}, and $u_{L/\mbb{Q}_2}=3$.
    \end{itemize}

    \item[{\rm (2)}]
    Suppose that $v_2(\Delta)$ is even.
    \begin{itemize}
        \item[{\rm (2.1)}]
        If $\widetilde{\Delta}\equiv-1\pmod{4}$, then
        $L=FL_i$ for some $i\in\{1,2\}$ in
        Table~\ref{table:degree2}, and $u_{L/\mbb{Q}_2}=2$.

        \item[{\rm (2.2)}]
        If $v_2(\Delta)\in\{6,8,12,14\}$ and
        $3v_2(c_4)\geq 8+v_2(\Delta)$, then
        $L\simeq F[X]/(X^3+2)$ and $u_{L/\mbb{Q}_2}=1$.

        \item[{\rm (2.3)}]
        Otherwise, $L=F$ and $u_{L/\mbb{Q}_2}=0$.
    \end{itemize}
\end{itemize}
\end{theorem}

\subsection{Examples}
The following table collects selected explicit elliptic curves illustrating the cases in 
Theorems~\ref{thm:PotentialMulti}--\ref{thm:Additive:p=2}.
\begin{table}[H]
\begin{center}
\begin{tabular}{|c|c||c|c||c|c|} \hline 
\multicolumn{6}{|c|}{Theorem \ref{thm:PotentialMulti}} \\ \hline
(1.1) & $E=$\href{https://www.lmfdb.org/EllipticCurve/Q/30a1/}{30a1} & 
(1.2) & $E=$\href{https://www.lmfdb.org/EllipticCurve/Q/30a3/}{30a3} & 
(1.3) & $E=$\href{https://www.lmfdb.org/EllipticCurve/Q/129b1/}{129b1} \\ \hline
(2.1) & $E=$\href{https://www.lmfdb.org/EllipticCurve/Q/150c1/}{150c1} & 
(2.2) & $E=$\href{https://www.lmfdb.org/EllipticCurve/Q/42a1/}{42a1} & 
(2.3) & $E=$\href{https://www.lmfdb.org/EllipticCurve/Q/33a1/}{33a1} \\ \hline
\end{tabular}

\vspace{1em}

\begin{tabular}{|c|c||c|c||c|c||c|c|} \hline 
\multicolumn{8}{|c|}{Theorem \ref{thm:Multiplicative:p=2}} \\ \hline
$L=\mbb{Q}_2$ & $E=$\href{https://www.lmfdb.org/EllipticCurve/Q/30a6/}{30a6} & 
$L=\mbb{Q}_4$ & $E=$\href{https://www.lmfdb.org/EllipticCurve/Q/78a3/}{78a3} & 
$L=L_1$ & $E=$\href{https://www.lmfdb.org/EllipticCurve/Q/42a4/}{42a4} & 
$L=L_2$ & $E=$\href{https://www.lmfdb.org/EllipticCurve/Q/78a4/}{78a4} \\ \hline
$L=L_3$ & $E=$\href{https://www.lmfdb.org/EllipticCurve/Q/50a1/}{50a1} & 
$L=L_4$ & $E=$\href{https://www.lmfdb.org/EllipticCurve/Q/26a2/}{26a2} & 
$L=L_5$ & $E=$\href{https://www.lmfdb.org/EllipticCurve/Q/14a5/}{14a5} & 
$L=L_6$ & $E=$\href{https://www.lmfdb.org/EllipticCurve/Q/30a7/}{30a7} \\ \hline
\end{tabular}

\vspace{1em}

\resizebox{\textwidth}{!}{%
\begin{tabular}{|c|c||c|c||c|c|} \hline 
\multicolumn{6}{|c|}{Theorem \ref{thm:PotentialGood}} \\ \hline
$(1.1)$ & $E=$\href{https://www.lmfdb.org/EllipticCurve/Q/147c1/}{147c1} & 
$(1.2)$ & $E=$\href{https://www.lmfdb.org/EllipticCurve/Q/294b1/}{294b1} & 
$(1.3)$ & $E=$\href{https://www.lmfdb.org/EllipticCurve/Q/50b1/}{50b1} \\ \hline
$(1.4)$ & $E=$\href{https://www.lmfdb.org/EllipticCurve/Q/225a1/}{225a1} & 
$(1.5)$ & $E=$\href{https://www.lmfdb.org/EllipticCurve/Q/1210f1/}{1210f1} & 
$(2.1)$ & $E=$\href{https://www.lmfdb.org/EllipticCurve/Q/150a3/}{150a3} ($p=5$) \\ 
\hline
$(2.2)$ & $E=$\href{https://www.lmfdb.org/EllipticCurve/Q/150a1/}{150a1} ($p=5$)& 
$(2.3)$ & $E=$\href{https://www.lmfdb.org/EllipticCurve/Q/49a1/}{49a1} ($p=7$)& 
$(2.4)$ & $E=$\href{https://www.lmfdb.org/EllipticCurve/Q/1274j1/}{1274j1} ($p=7$) \\ \hline 
$(3.1)$ & $E=$\href{https://www.lmfdb.org/EllipticCurve/Q/294d1/}{294d1} ($p=7$)& 
$(3.2)$ & $E=$\href{https://www.lmfdb.org/EllipticCurve/Q/637a1/}{637a1} ($p=7$)& 
$(3.3)$ & $E=$\href{https://www.lmfdb.org/EllipticCurve/Q/50a1/}{50a1} ($p=5$) \\ \hline 
$(3.4)$ & $E=$\href{https://www.lmfdb.org/EllipticCurve/Q/1089a1/}{1089a1} ($p=11$)& 
$(3.5)$ & $E=$\href{https://www.lmfdb.org/EllipticCurve/Q/350f1/}{350f1} ($p=5$)& 
$(4.1)$ & $E=$\href{https://www.lmfdb.org/EllipticCurve/Q/1323b1/}{1323b1} ($p=7$) \\ \hline 
$(4.2)$ & $E=$\href{https://www.lmfdb.org/EllipticCurve/Q/833a2/}{833a2} ($p=7$)& 
$(4.3)$ & $E=$\href{https://www.lmfdb.org/EllipticCurve/Q/350d1/}{350d1} ($p=5$)& 
$(5.1)$ & $E=$\href{https://www.lmfdb.org/EllipticCurve/Q/147b1/}{147b1} ($p=7$) \\ \hline 
$(5.2)$ & $E=$\href{https://www.lmfdb.org/EllipticCurve/Q/294a1/}{294a1} ($p=7$)& 
$(5.3)$ & $E=$\href{https://www.lmfdb.org/EllipticCurve/Q/50a3/}{50a3} ($p=5$)& 
$(5.4)$ & $E=$\href{https://www.lmfdb.org/EllipticCurve/Q/225b1/}{225b1} ($p=5$) \\ \hline 
$(5.5)$ & $E=$\href{https://www.lmfdb.org/EllipticCurve/Q/1210m1/}{1210m1} ($p=11$)& 
$(6.1)$ & $E=$\href{https://www.lmfdb.org/EllipticCurve/Q/150b3/}{150b3} ($p=5$)& 
$(6.2)$ & $E=$\href{https://www.lmfdb.org/EllipticCurve/Q/150b1/}{150b1} ($p=5$) \\ \hline 
$(6.3)$ & $E=$\href{https://www.lmfdb.org/EllipticCurve/Q/49a3/}{49a3} ($p=7$)& 
$(6.4)$ & $E=$\href{https://www.lmfdb.org/EllipticCurve/Q/1274m1/}{1274m1} ($p=7$)& 
$(7.1)$ & $E=$\href{https://www.lmfdb.org/EllipticCurve/Q/294e1/}{294e1} ($p=7$) \\ \hline 
$(7.2)$ & $E=$\href{https://www.lmfdb.org/EllipticCurve/Q/441a1/}{441a1} ($p=7$)& 
$(7.3)$ & $E=$\href{https://www.lmfdb.org/EllipticCurve/Q/50b3/}{50b3} ($p=5$)& 
$(7.4)$ & $E=$\href{https://www.lmfdb.org/EllipticCurve/Q/1089b1/}{1089b1} ($p=11$) \\ \hline 
$(7.5)$ & $E=$\href{https://www.lmfdb.org/EllipticCurve/Q/350e1/}{350e1} ($p=5$)& 
 & & 
 &  \\ \hline 
\end{tabular}%
}

\vspace{1em}

\resizebox{\textwidth}{!}{%
\begin{tabular}{|c|c||c|c|} \hline 
\multicolumn{4}{|c|}{Theorem \ref{thm:Additive:p=3}} \\ \hline
(1.1) & $E=$\href{https://www.lmfdb.org/EllipticCurve/Q/54b1/}{54b1} & 
(1.2) & $E=$\href{https://www.lmfdb.org/EllipticCurve/Q/36a4/}{36a4} \\ \hline 
(1.3) & $E=$\href{https://www.lmfdb.org/EllipticCurve/Q/441d1/}{441d1} &
(2.1) & $E=$\href{https://www.lmfdb.org/EllipticCurve/Q/27a2/}{27a2} 
($K \simeq \mbb{Q}_3[X]/(X^3+3)$) \\ \hline
(2.1) &  
$E=$\href{https://www.lmfdb.org/EllipticCurve/Q/189a1/}{189a1} 
($K \simeq \mbb{Q}_3[X]/(X^3+12)$) &
(2.1) & $E=$\href{https://www.lmfdb.org/EllipticCurve/Q/324a1/}{324a1} 
($K \simeq \mbb{Q}_3[X]/(X^3+21)$) \\ \hline
\end{tabular}%
}

\vspace{1em}

\resizebox{\textwidth}{!}{%
\begin{tabular}{|c|c||c|c||c|c|} \hline 
\multicolumn{6}{|c|}{Theorem \ref{thm:Additive:p=2}} \\ \hline
(1.1) & $E=$\href{https://www.lmfdb.org/EllipticCurve/Q/56a4/}{56a4} ($L=L_3$) & 
(1.1) & $E=$\href{https://www.lmfdb.org/EllipticCurve/Q/96a3/}{96a3} ($L=L_4$) & 
(1.1) & $E=$\href{https://www.lmfdb.org/EllipticCurve/Q/56a3/}{56a3} ($L=L_5$) 
\\ \hline
(1.1) & $E=$\href{https://www.lmfdb.org/EllipticCurve/Q/120b3/}{120b3} ($L=L_6$) & 
(2.1) & $E=$\href{https://www.lmfdb.org/EllipticCurve/Q/20a3/}{20a3} ($L=L_1$) & 
(2.1) & $E=$\href{https://www.lmfdb.org/EllipticCurve/Q/36a4/}{36a4} ($L=L_2$) \\ \hline
(2.2) & $E=$\href{https://www.lmfdb.org/EllipticCurve/Q/44a2/}{44a2} 
($L\simeq F[X]/(X^3+2)$) & 
(2.3) & $E=$\href{https://www.lmfdb.org/EllipticCurve/Q/20a4/}{20a4} & 
 & \\ \hline
\end{tabular}%
}

\end{center}
\caption{Examples for the cases in  Theorems~\ref{thm:PotentialMulti}--\ref{thm:Additive:p=2}}\label{table:examples}
\end{table}
\section{Ramification theory}\label{sect:Ramification}
Let $K$ be a finite extension of $\mbb{Q}_p$ and $\overline{K}$ 
a fixed algebraic closure, and $v_K$ the valuation on $\overline{K}$ 
normalized by $v_K(K^{\times})=\mbb{Z}$. 

\subsection{Ramification breaks}
In this section, we give a review of the classical ramification 
theory for finite separable extensions of $K$ studied in \cite[Appendice]{Del84} and 
\cite{Hel91}. 
Let $L$ be a finite extension of $K$.
Set $\Gamma(L/K):=\mathrm{Hom}_K(L,\overline{K})$.
The function $\mbf{i}_{L/K}$ on $\Gamma(L/K)$ is defined by
\[
\mbf{i}_{L/K}(\sigma)=\underset{a \in \mcal{O}_L}{\mathrm{inf}}
v_L(\sigma(a)-a)=v_L(\sigma x-x), \quad \sigma\in\Gamma(L/K),
\]
where $\mcal{O}_L=\mcal{O}_K[x]$
(cf.\ \cite[Chapter III, Proposition 12]{Ser68}), and $v_L$ is the
valuation on $\overline{K}$ normalized by $v_L(L^{\times})=\mbb{Z}$.
Since the different $\mfrak{D}_{L/K}$ is generated by $f'(x)$,
where $f$ is the minimal polynomial of $x$ over $K$, we have
\[
v_L(\mfrak{D}_{L/K})=\sum_{\sigma\not=1}\mbf{i}_{L/K}(\sigma).
\]
For any real number $m\geq-1$, the {\it lower ramification set}
$\Gamma(L/K)_m$ is defined by 
\[
\Gamma(L/K)_m=\{\sigma \in \Gamma(L/K) \, |\, \mbf{i}_{L/K}(\sigma) 
\geq m+1 \}.
\]
The {\it Herbrand function} $\varphi_{L/K}$ of $L/K$ is defined by 
\[
\varphi_{L/K}(m)=\frac{1}{e_{L/K}}\int_0^m \sharp \Gamma(L/K)_y \,dy
\]
for $m>0$ and $\varphi_{L/K}(m)=m$ for $-1 \leq m \leq 0$. 
Let $\psi_{L/K}$ be its inverse function. 
For any real number $m \geq -1$, 
the {\it upper ramification set} $\Gamma(L/K)^m$ is defined by 
$\Gamma(L/K)^m=\Gamma(L/K)_{\psi_{L/K}(m)}$. 
If $L/K$ is a Galois extension, then $\Gamma(L/K)^m$ 
coincides with the upper ramification group 
$\mathrm{Gal}(L/K)^m$ defined in \cite[Chapter IV]{Ser68}. 
\begin{proposition}[\cite{Del84}, Proposition A.6.1]
Let $L$ be a finite extension of $K$ and 
$\widetilde{L}$ the Galois closure of $L$ over $K$. 
Then, for any real number $m \geq -1$, 
the following statements are equivalent:
\begin{itemize}
\item[\rm{(i)}] $\Gamma(L/K)^m=1$
\item[\rm{(ii)}] $\mrm{Gal}(\widetilde{L}/K)^m=1$
\item[\rm{(iii)}] $\displaystyle{\sum_{\sigma \not= 1}}  
v_K(\sigma x -x) + \underset{\sigma \not= 1}{\rm{sup}} \, 
v_K(\sigma x - x)<m+1$. 
\end{itemize}
\end{proposition}
In this paper, we use the upper ramification filtration shifted by one as follows. 
For a finite extension $L$ of $K$ and a real number $m \geq 0$, 
we set $\Gamma(L/K)^{(m)}=\Gamma(L/K)^{m-1}$.
If $L/K$ is a Galois extension with Galois group $G$, 
then $\Gamma(L/K)^{(m)}=G^{(m)}$ for $m \geq 0$, 
where $G^{(m)}$ is the upper ramification group 
defined in 
\cite[Section 1]{Fon85} (cf. Remarques 1.2 in {\it ibid.}). 
We define the {\it maximal upper ramification break} of $L/K$ by 
\[
u_{L/K}=\mrm{sup} \, \{ m \in \mbb{R}_{\geqslant 0} \, |\, \Gamma(L/K)^{(m)}\not=1\}.
\]
We also define the {\it maximal lower ramification break} of $L/K$ by 
$i_{L/K}=\underset{\sigma \not= 1}{\rm{sup}} \, v_K(\sigma x - x)$. 
For $L=K$, we set $u_{L/K}=0$ by convention. 
\begin{corollary}
\label{cor:GaloisClosure}
    Let $L$ be a finite extension of $K$ and $\widetilde{L}$ the Galois closure 
    of $L$ over $K$. Then we have $u_{L/K}=u_{\widetilde{L}/K}$ and 
    \[
    u_{L/K}=v_K(\mfrak{D}_{L/K})+i_{L/K}.
    \]
\end{corollary}
\begin{proposition}
\label{prop:break}
Let $L$ be a finite extension of $K$. 
\begin{itemize}
    \item[\rm{(i)}] $L/K$ is unramified if and only if $u_{L/K}=0$.
    \item[\rm{(ii)}] $L/K$ is tamely ramified if and only if $u_{L/K} \leq 1$. 
\end{itemize}
\end{proposition}
\begin{proof}
By Corollary \ref{cor:GaloisClosure}, we may assume $L/K$ is a Galois extension. 
Let $G$ be the Galois group of $L/K$. 
It is well known that, for any real number $0<m \leq 1$, 
$G^{(m)}$ is the inertia subgroup of $G$ and 
$\cup_{m'>1} G^{(m')}$ is the wild inertia subgroup of $G$ 
(cf. \cite[Chapter IV]{Ser68}). 
These equalities give the results. 
\end{proof}
\begin{proposition}
\label{prop:composite}
Let $L$ and $M$ be two finite extensions of $K$. Then we have
\[ u_{LM/K}=\max \{ u_{L/K},u_{M/K} \}.\]
\end{proposition}
\begin{proof}
By Corollary \ref{cor:GaloisClosure}, 
it suffices to show the case where $L$ and $M$ are Galois extensions of $K$.
In this case, the assertion is a basic property of upper ramification breaks 
(cf. \cite{Yos10}, Remark 3.4 for the proof).
\end{proof}
\begin{proposition}[\cite{Del84}, Proposition A.6.2]
\label{prop:beta}
Let $L$ be a finite extension of $K$ and 
$f$ the minimal polynomial of $x$ over $K$, 
where $x$ is a generator of $\mathcal{O}_L$ over $\mathcal{O}_K$. 
For a real number $m \geq 0$, the following statements are equivalent:
\begin{itemize}
\item[\rm{(i)}] $u_{L/K}<m$.
\item[\rm{(ii)}] For any $y \in \overline{K}$ with $v_K(f(y)) \geq m$, 
there exists $\sigma_0 \in \Gamma(L/K)$ such that 
$v_K(y-\sigma_0x)>v_K(y-\sigma x)$ for every 
$\sigma\in\Gamma(L/K)\setminus\{\sigma_0\}$. 
\end{itemize}
\end{proposition}
Let $E_K^e$ be the set of Eisenstein polynomials over $K$ of degree $e$. 
For $f(X), g(X) \in E_K^e$, we put $v_K(f,g)=v_K(f(\pi_g))$, where $\pi_g$ is a root of $g$. 
Then $v_K(f,g)$ defines an ultrametric on $E_K^e$ 
(\cite{Kra66} or \cite{Pau01} for proofs). 
\begin{proposition}[\cite{Pau01}, Lemma 4.2]
\label{prop:ultrametric}
Let $f(X)=X^e+a_{e-1}X^{e-1}+\cdots+a_1X+a_0$, 
$g(X)=X^e+b_{e-1}X^{e-1}+\cdots+b_1X+b_0 \in E_K^e$. 
Then we have 
\[
v_K(f,g)=\underset{i}{\mathrm{min}}\left\{ v_K(a_i-b_i)+\frac{i}{e} \right\}.
\]
\end{proposition}
\begin{proposition}
\label{prop:criteria}
For $f, g \in E_K^e$, we put $L=K[X]/(f)$ and $M=K[X]/(g)$. 
If $v_K(f,g)>u_{L/K}$, we have a $K$-isomorphism $L \simeq M$. 
\end{proposition}
\begin{proof}
By assumption, we have $v_K(f(\pi_g))>u_{L/K}$ for a root of $g$. 
According to Proposition \ref{prop:beta}, there exists a root $\pi_f$ of $f$ 
such that $v_K(\pi_g-\pi_f)>v_K(\pi_g-\sigma \pi_f)$ for any $\sigma \not= 1$. 
Take $\sigma_1 \not=1$ such that $i_{L/K}=v_K(\sigma_1 \pi_f-\pi_f)$. 
Then we have
\[
i_{L/K}=v_K(\sigma_1\pi_f-\pi_f)=v_K(\sigma_1\pi_f-\pi_g+\pi_g-\pi_f)=v_K(\sigma_1\pi_f-\pi_g).
\]
Thus we have $v_K(\pi_g-\pi_f)>i_{L/K}$. By Krasner's Lemma, we have 
$L \simeq K(\pi_f) \subset K(\pi_g) \simeq M$. 
Since both extensions $L$ and $M$ have the same degree $e$, 
we have a $K$-isomorphism $L \simeq M$. 
\end{proof}
\begin{proposition}\label{prop:BoundRamificationBreak}
Let $e_K$ be the absolute ramification index of $K$. 
Let $L$ be a finite extension of $K$ with ramification index $e$. 
Then we have the inequality
\[ u_{L/K} \leq 1+e_K \left( v_p(e)+\frac{1}{p-1} \right). \]
\end{proposition}
\begin{proof}
If $L/K$ is unramified, then the assertion is immediate. Otherwise,
let $K_0$ be the maximal unramified subextension of $L/K$.
Since $K_0/K$ is unramified, we have
$u_{L/K}=u_{L/K_0}$, $e_{K_0}=e_K$, and
$e(L/K_0)=e(L/K)=e$. Thus, replacing $K$ by $K_0$, we may assume
that $L/K$ is totally ramified.

Let $\pi_L$ be a prime element of $L$. Then
$\mcal{O}_L=\mcal{O}_K[\pi_L]$, and hence
\[
i_{L/K}
 =\sup_{\sigma\neq1}v_K(\sigma\pi_L-\pi_L).
\]
By \cite[Chap.~III, Sect.~7, Proposition~13, Remarks~(1)]{Ser68},
we have
\[
v_K(\mfrak{D}_{L/K})
 \leq \frac{e-1}{e}+v_K(e).
\]
By \cite[Chap.~IV, Sect.~2, Exercise~3(c)]{Ser68}, we have
\[
\sup_{\sigma\neq1}v_L(\sigma\pi_L-\pi_L)-1
 \leq \frac{e_L}{p-1},
\]
where $e_L$ is the absolute ramification index of $L$. Therefore,
\[
i_{L/K}
 \leq \frac{e_K}{p-1}+\frac{1}{e}.
\]
By Corollary~\ref{cor:GaloisClosure}, we obtain
\[
u_{L/K}
 \leq \frac{e-1}{e}+v_K(e)
      +\frac{e_K}{p-1}+\frac{1}{e}
 =1+e_K\left(v_p(e)+\frac{1}{p-1}\right).
\]
\end{proof}
\subsection{Extensions of degree \texorpdfstring{$p>2$}{p > 2}}
In this section, we review a classification of 
the isomorphism classes of ramified extensions of degree $p$
over $K$ 
following \cite{Ama71}. 
After that, we prove some results on their Galois closure. 
These results will be applied to fields of definition of $p$-torsion points on elliptic curves. 
Throughout this section, we assume that 
$p$ is an odd prime.
We fix a prime element $\pi_K$ of $K$ and a complete system $\mfrak{R}_K$
of representatives of $\mcal{O}_K$ modulo $\mbf{m}_K$ such that 
$0 \in \mfrak{R}_K$. 
For any $x, y \in \overline{K}$, the congruence $x \equiv y$ means that 
$v_K(x-y)>0$. 
Let $L$ be a ramified extension of degree $p$ over $K$. 
A prime element $\pi_L$ of $L$ satisfies an Eisenstein equation $f(X)=0$, where 
\[
f(X)=X^p-\sum_{i=1}^{p-1}a_iX^i-\pi_Ku \in E_K^p.
\]
Replacing $\pi_L$ by a suitable $k \pi_L$ with $k \in K$, we may assume that $\pi_L^p \equiv N_{L/K} \pi_L \equiv \pi_K$, where $N_{L/K}$ is the norm map. 
This congruence implies $u \equiv 1$. 
Throughout this section, we assume that the Eisenstein polynomial
$f$ and the prime element $\pi_L$ are chosen so that
$\pi_L^p\equiv\pi_K$ and $u\equiv1$. 
For $f(X) \in E_K^p$ such that $f(X)=0$ has a root in $L$, 
we define the {\it type} of $L$ as follows:
\begin{itemize}
\item[(i)] In case $\underset{i}{\min}\, v_K(a_i)=m \leq e_K$, we denote by $\lambda$ 
the least integer satisfying $v_K(a_{\lambda})=m$, and take $\omega \not= 0$ in $\mfrak{R}_K$ 
such that $a_{\lambda} \equiv \omega \pi_K^m$. 
In this case, we say that $f(X)$ is of type $\langle \lambda,m,\omega \rangle$.
\item[(ii)] In case $v_K(a_i) \geq e_K+1$ for $1 \leq i \leq p-1$, we set $\lambda=0$ and 
$m=e_K+1$. 
In this case, we say that $f(X)$ is of type $\langle 0 \rangle$. 
\end{itemize}
The type of $L$ does not depend on the choice of $f(X)$ (\cite[Theorem 1]{Ama71}). 
Hence the type of $L$ is an invariant of the isomorphism class of $L$ over $K$. 
We distinguish the case (i) and (ii) by $\lambda \not=0$ and $\lambda=0$, respectively. 
\begin{proposition}
    \label{prop:BreakDegreeP}
    For any ramified extension $L$ of degree $p$ over $K$, 
    we have the equalities
    \[ v_K(\mfrak{D}_{L/K})=m+\frac{\lambda-1}{p},\quad 
u_{L/K}=m+\frac{m+\lambda-1}{p-1}.\]
    In particular, the inequality $u_{L/K} \geq 1+\frac{1}{p-1}$ holds. 
\end{proposition}
\begin{proof}
For any $\sigma \in \Gamma(L/K)$ with $\sigma\not=1$, all the values $v_K(\sigma \pi_L -\pi_L)$ are the same 
by \cite[Lemma 1]{Ama71}. 
Hence we have $v_K(\mathfrak{D}_{L/K})=(p-1)i_{L/K}$ and $u_{L/K}=p \cdot i_{L/K}$. 
Since $f'(\pi_L)=p\pi_L^{p-1}-\sum_{i=1}^{p-1} i a_i \pi_L^{i-1}$, 
we have 
\[
v_K(\mfrak{D}_{L/K})=\min\, 
\{ v_K(p\pi_L^{p-1}),\dots,v_K(ia_i\pi_L^{i-1}),\dots, v_K(a_1)\}=m+\frac{\lambda-1}{p}.
\]
The result follows.
\end{proof}
\begin{proposition}[\cite{Ama71}, Theorem 4]
    \label{prop:CanonicalPolynomial}
    Let $L$ be a ramified extension of degree $p$ over $K$. 
    Then there 
    exists a prime element of $L$ which 
    satisfies an Eisenstein equation of the following form:
    \begin{itemize}
        \item[{\rm (i)}] $X^p-\omega \pi_K^mX^{\lambda}-\pi_Ku=0$, if $L$ 
        is of type $\langle \lambda,m,\omega \rangle$,
        \item[{\rm (ii)}] $X^p-\pi_Ku=0$, if $L$ is of type $\langle 0 \rangle$.
    \end{itemize}
\end{proposition}
\begin{proposition}[\cite{Ama71}, Theorem 2 and 3]\label{prop:GaloisCriteria}
    Let $L$ be a ramified extension of degree $p$ over $K$ of 
    type $\langle \lambda,m,\omega \rangle$. 
    Then $L/K$ is a cyclic extension if and only if $m+\lambda-1 \equiv 0 \pmod{p-1}$ and 
    $K$ contains an element $s$ such that $s^{p-1}=\lambda \omega$. 
    If $L/K$ is not a Galois extension, then the Galois closure $\widetilde{L}$ of $L/K$ is given by $K(\pi_L,\gamma)$, where $\gamma$ is an element of $\overline{K}$ 
    such that $\gamma^{p-1}=\lambda \omega \pi_K^{m+\lambda-1}$. 
\end{proposition}
We define $\mfrak{M}(\lambda,m)$ to be the set of all units in 
$\mathcal{O}_K$ of the form 
\[
1+\sum_i \alpha_i \pi_K^i, \quad \alpha_i \in \mfrak{R}_K,
\]
where the summation is taken over all integers $i$ such that 
\[
1 \leq i \leq m-1 + \left\lfloor \frac{m+\lambda-1}{p-1} \right\rfloor \quad {\rm and} \quad  
i \not\equiv -\lambda \pmod{p}.
\]
\begin{proposition}[\cite{Ama71}, Theorem 7]
    \label{prop:ClassificationDegreeP}
    Each isomorphism class of ramified non-Galois extensions 
    of degree $p$ over $K$ of type 
    $\langle \lambda,m,\omega \rangle$ is defined by the roots of exactly one equation of the form
    \[
    X^p-\omega \pi_K^mX^{\lambda}-\pi_Ku=0 \quad {\rm with} \quad u \in \mfrak{M}(\lambda,m).
    \]
    Thus we obtain a one-to-one correspondence between the set $\mfrak{M}(\lambda,m)$ and 
    the set of isomorphism classes of type $\langle \lambda,m,\omega \rangle$. 
\end{proposition}
\begin{corollary}
\label{cor:ClassificationLambda}
Let $F$ be a finite unramified extension of $\mbb{Q}_p$, and fix an integer 
$\lambda$ with $1\leq\lambda\leq p-2$. Every degree-$p$ ramified extension 
$L/F$ of type $\langle\lambda,1,\omega\rangle$ is defined by the roots 
of exactly one equation of the form
    \[
    X^p-\omega\pi_FX^{\lambda}-\pi_F=0 \quad {\rm with} \quad
    \omega\in\mfrak{R}_F\setminus\{0\}.
    \]
    Thus, as $\omega$ ranges over $\mfrak{R}_F\setminus\{0\}$, these 
    equations are in one-to-one correspondence with the $F$-isomorphism 
    classes of such extensions.
\end{corollary}
\begin{proof}
Since $F$ is unramified, we have $m=1$. Hence 
$\mfrak{M}(\lambda,1)=\{1\}$, and the assertion follows from 
Proposition~\ref{prop:ClassificationDegreeP}.
\end{proof}
\begin{proposition}\label{prop:breaklambda}
Let $F$ be a finite unramified extension of $\mbb{Q}_p$ and 
$L$ a degree $p$ ramified extension of $F$. 
Suppose $L$ is of type $\langle \lambda,m,\omega \rangle$ and 
the Galois closure $\widetilde{L}$ of $L/F$ equals $L(\mu_p)$, 
where $\mu_p$ is the group of $p$-th roots of unity. 
Then we have 
\[ u_{\widetilde{L}/F}=\frac{p-1+\lambda}{p-1},\quad v_p(\mfrak{D}_{\widetilde{L}/F})=\frac{p^2+(\lambda-1)p-\lambda-1}
{p(p-1)}.\]
\end{proposition}
\begin{proof}
Since $F/\mbb{Q}_p$ is unramified, we have $m=1$. 
By Proposition \ref{prop:composite} and Proposition \ref{prop:BreakDegreeP}, 
we have 
$u_{\widetilde{L}/F}=\max \{ u_{L/F}, u_{F(\mu_p)/F} \} = 
u_{L/F} =1+\frac{\lambda}{p-1}$. 
Corollary \ref{cor:GaloisClosure} and 
Proposition \ref{prop:BreakDegreeP} show 
$i_{L/F}=u_{L/F}-v_p(\mfrak{D}_{L/F})=\frac{\lambda}{p-1}-
\frac{\lambda-1}{p}$. 
Put $\pi_{\widetilde{L}}=\frac{\zeta_p-1}{\pi_L}$, where $\zeta_p$ 
is a $p$-th primitive root of unity and 
$\pi_L$ is a prime element of $L$.
Then $\pi_{\widetilde{L}}$ is a prime element of $\widetilde{L}$. 
Since $W=\mrm{Gal}(\widetilde{L}/F(\mu_p))$ coincides with 
the wild inertia subgroup of $G=\mrm{Gal}(\widetilde{L}/F)$, 
\begin{align*}
i_{\widetilde{L}/F}&=\underset{1 \not= \sigma \in G}{\sup} 
v_p(\sigma \pi_{\widetilde{L}}-\pi_{\widetilde{L}})
=\underset{1 \not= \sigma \in W}{\sup} 
v_p(\sigma \pi_{\widetilde{L}}-\pi_{\widetilde{L}}) \\ 
&=\underset{1 \not= \sigma \in W}{\sup} v_p(\frac{\zeta_p-1}{\sigma\pi_L}
- \frac{\zeta_p-1}{\pi_L}) \\ 
&=v_p(\zeta_p-1)+ \underset{1 \not= \sigma \in W}{\sup} v_p(
\frac{\pi_L-\sigma\pi_L}{\sigma\pi_L\cdot\pi_L}) \\
&= \frac{1}{p-1} + i_{L/F} - \frac{2}{p} = \frac{\lambda+1}{p(p-1)}.
\end{align*}
Thus we have 
\begin{align*}
v_p(\mfrak{D}_{\widetilde{L}/F})
&=u_{\widetilde{L}/F}-i_{\widetilde{L}/F} = 
1+\frac{\lambda}{p-1} - \frac{\lambda+1}{p(p-1)} \\ 
&=\frac{p^2+(\lambda-1)p-\lambda
-1}{p(p-1)}
\end{align*}
\end{proof}
\begin{remark}
Since the ramification index of $\widetilde{L}/F$ is $p(p-1)$, we have
\[
v_{\widetilde{L}}(\mfrak{D}_{\widetilde{L}/F})
=p^2+(\lambda-1)p-\lambda-1.
\]
Hence
\[
p^2-2\leq v_{\widetilde{L}}(\mfrak{D}_{\widetilde{L}/F})
\leq 2p^2-3p,
\]
since $1\leq\lambda\leq p-1$.
\end{remark}
\begin{proposition}\label{prop:break0}
Let $F$ be finite unramified extension of $\mbb{Q}_p$ 
and $L$ a degree $p$ ramified extension of $F$. 
Suppose $L$ is of type $\langle 0 \rangle$, so that 
the Galois closure $\widetilde{L}$ of $L/F$ equals $L(\mu_p)$.
Then we have
\[ u_{\widetilde{L}/F}=\frac{2p-1}{p-1}, \quad 
v_p(\mfrak{D}_{\widetilde{L}/F})=\frac{2p^2-2p-1}{p(p-1)}. \]
\end{proposition}
\begin{proof}
By Proposition \ref{prop:CanonicalPolynomial} (ii), 
we have an isomorphism $L \simeq F[x]/(x^p-\pi_F)$ for 
some prime element $\pi_F$ of $F$. 
By Proposition \ref{prop:composite}, we have 
$u_{\widetilde{L}/F}=\max \{ u_{L/F}, u_{F(\mu_p)/F} \}
=u_{L/F}=2+\frac{1}{p-1}$. 
Put $\pi_{\widetilde{L}}=\frac{\zeta_p-1}{\pi_F^{1/p}}$, 
where $\zeta_p$ is a $p$-th primitive root of unity. 
Then $\pi_{\widetilde{L}}$ is a prime element of $\widetilde{L}$. 
Since $W=\mrm{Gal}(\widetilde{L}/F(\mu_p))$ coincides with 
the wild inertia subgroup of $G=\mrm{Gal}(\widetilde{L}/F)$, 
\begin{align*}
i_{\widetilde{L}/F}&=\underset{1 \not= \sigma \in G}{\sup} 
v_p(\sigma \pi_{\widetilde{L}}-\pi_{\widetilde{L}}) 
=\underset{1 \not= \sigma \in W}{\sup} 
v_p(\sigma \pi_{\widetilde{L}}-\pi_{\widetilde{L}}) 
\\ 
&=\sup_{1\leq i\leq p-1} v_p(\frac{\zeta_p-1}{\pi_F^{1/p}\zeta_p^i} - 
\frac{\zeta_p-1}{\pi_F^{1/p}}) \\
&=v_p(\zeta_p-1)+\sup_{1\leq i\leq p-1} v_p(\frac{1-\zeta_p^i}{\pi_F^{1/p}\zeta_p^i}) 
=\frac{1}{p-1}+\frac{1}{p-1}-\frac{1}{p} \\
&=\frac{p+1}{p(p-1)}.
\end{align*}
Thus we have 
\begin{align*}
v_p(\mfrak{D}_{\widetilde{L}/F})&=u_{\widetilde{L}/F}-i_{\widetilde{L}/F}=2+\frac{1}{p-1} - \frac{p+1}{p(p-1)} \\
&=\frac{2p^2-2p-1}{p(p-1)}.
\end{align*}
\end{proof}
\begin{remark}
Comparing Proposition \ref{prop:breaklambda} with Proposition \ref{prop:break0}, 
the invariant $\lambda$ of $L$ is determined only by the valuation 
$v_p(\mfrak{D}_{\widetilde{L}/F})$. 
\end{remark}
We denote by $\mfrak{U}$ the group of all units $u$ in $\mcal{O}_K$ 
such that $u \equiv 1$. 
Let $\mfrak{M}(0)$ be a fixed representative system of $\mfrak{U}$ modulo 
$\mfrak{U}^p$. 
\begin{proposition}[\cite{Ama71}, Theorem 6]\label{prop:ClassificationType0}
Each isomorphism class of ramified extensions of degree $p$ over $K$ 
of type $\langle 0 \rangle$ is defined by the roots of exactly one equation of the form 
\[ X^p-\pi_K u=0 \quad \mrm{with} \quad u \in \mfrak{M}(0). \]
Thus we obtain a one-to-one correspondence between the set $\mfrak{M}(0)$ 
and the isomorphism classes of such extensions $L/K$. 
\end{proposition}
\begin{corollary}\label{cor:ClassificationType0}
Let $F$ be a finite unramified extension of degree $f$ over $\mbb{Q}_p$. 
Then each isomorphism class of degree $p$ ramified extensions $L/F$ 
of type $\langle 0 \rangle$ is defined by the roots of exactly 
one equation of the form 
\[ X^p+p(1+up)=0 \quad \mrm{with} \quad u \in \mfrak{R}_F. \]
Thus we obtain a one-to-one correspondence between the set $\mfrak{R}_F$ 
and the isomorphism classes of such extensions $L/F$. 
\end{corollary}
\subsection{Extensions of degree \texorpdfstring{$p=2$}{p = 2}}
In this subsection, we present a list of the quadratic extensions $L$ 
of $\mbb{Q}_2$. 
Table \ref{table:degree2} presents such a list, which stores 
a defining polynomial, the residue degree $f$, the ramification index $e$, 
the valuation of the different $v_p(\mfrak{D}_{L/\mbb{Q}_2})$ and 
the upper ramification break $u_{L/\mbb{Q}_2}$. 
\begin{table}[H]
\begin{center}
\begin{tabular}{|c|c|c|c|c|c|} \hline 
     $L$ & Polynomial & $f$ & $e$ & $v_p(\mfrak{D}_{L/\mbb{Q}_2})$ & 
     $u_{L/\mbb{Q}_2}$ \\ \hline
  $\mbb{Q}_4$ & $X^2+X+1$ & $2$ & $1$ & $0$ & $0$ \\ \hline
  $L_1$ & $X^2+2X+2$ & $1$ & $2$ & $1$ & $2$ \\ \hline
  $L_2$ & $X^2+2X+6$ & $1$ & $2$ & $1$ & $2$ \\ \hline
  $L_3$ & $X^2+2$ & $1$ & $2$ & $3/2$ & $3$ \\ \hline
  $L_4$ & $X^2+10$ & $1$ & $2$ & $3/2$ & $3$ \\ \hline
  $L_5$ & $X^2+4X+2$ & $1$ & $2$ & $3/2$ & $3$ \\ \hline
  $L_6$ & $X^2+4X+10$ & $1$ & $2$ & $3/2$ & $3$ \\ \hline
\end{tabular}
\end{center}
\caption{The quadratic extensions of $\mbb{Q}_2$}\label{table:degree2}
\end{table}
\subsection{Ramification of \texorpdfstring{$p^n$}{p to the n}-torsion of elliptic curves}
Let $G_K$ be the absolute Galois group and, 
for any real number $m \geq 0$, $G_K^{(m)}$ the $m$-th 
upper ramification group defined in \cite{Fon85}. 
We compare with the notation $G_K^m$ defined in 
\cite[Chapter IV]{Ser68}, 
the equality $G_K^m=G_K^{(m+1)}$ holds. 
\begin{theorem}[\cite{Fon85}, Th\'eor\`eme 1]
\label{thm:BoundElliptic}
Let $n$ be an integer $\geq 1$ and $\mcal{G}$ a finite flat group scheme 
over $\mcal{O}_K$ killed by $p^n$. 
Let $H$ be the kernel of the action of $G_K$ 
on $\mcal{G}(\overline{K})$ and $L=\overline{K}^H$. 
Then we have $G_K^{(m)} \subset H$ for any $m>e_K \left( n+\frac{1}{p-1} \right)$ and 
$v_K(\mfrak{D}_{L/K})<e_K \left( n+\frac{1}{p-1} \right)$. 
\end{theorem}
\begin{corollary}
\label{cor:BreakBound}
Let $E$ be an elliptic curve over $K$ with good reduction. 
Put $L=K(E[p^n])$. 
Then we have 
\[
u_{L/K} \leq e_K \left( n+\frac{1}{p-1} \right).
\]
\end{corollary}
\begin{remark}
When $K=\mbb{Q}_p$ and $n=1$, 
this bound is sharp since $u_{L/\mbb{Q}_p}=1+\frac{1}{p-1}$ 
if $E$ has good ordinary reduction and $E[p]$ is wildly ramified (cf.\ Theorem \ref{cor:defferent}).  
\end{remark}
\begin{proposition}
Let $E$ be an elliptic curve over $K$ with $v_p(j)<0$. 
Put $L=K(E[p^n])$. Then we have 
\[ u_{L/K} \leq 1+e_K\left( n+\frac{1}{p-1} \right). \]
\end{proposition}
\begin{proof}
By the assumption $v_p(j)<0$, there is an isomorphism 
$E(\overline{K}) \simeq \overline{K}^{\times}/q^{\mathbb{Z}}$ over $T$ for some $q \in K^{\times}$, 
where $T/K$ is at most a quadratic extension. 
The $p^n$-torsion of $\overline{K}^{\times}/q^{\mbb{Z}}$
consists of the classes of roots of the equations $X^{p^n}=q^m$
$(m\in\mbb{Z})$. 
Hence, we have $L \subset T(\mu_{p^n},\sqrt[p^n]{q})=:M$.
By Propositions \ref{prop:composite} and \ref{prop:BoundRamificationBreak}, 
we have the inequalities
\[ u_{L/K} \leq u_{M/K}=\max \{ u_{T/K},u_{K(\mu_{p^n})/K},u_{K(\sqrt[p^n]{q})/K} \} \leq 1+e_K \left( 
n+\frac{1}{p-1} \right). \]
\end{proof}
\begin{remark}
When $p>2$ and $E$ has multiplicative reduction over $K$, this bound can also be obtained
by \cite[Theorem 1.1]{Hat09}. 
\end{remark}
\begin{proposition}\label{prop:RamificationBoundPotGood}
Suppose $p>3$.
Let $E$ be an elliptic curve with additive reduction and $v_p(j) \geq 0$. 
Put $L=K(E[p^n])$. 
Take the minimal extension $K'/K$ such that $E$ has good reduction over $K'$. 
Denote by $e_{K'/K}$ the ramification index of $K'/K$. 
Then we have
\[ u_{L/K} \leq 1+e_K\left( n+\frac{1}{p-1} \right) - \frac{1}{e_{K'/K}}. \]
In particular, the inequality $u_{L/K} \leq 1+e_K(n+\frac{1}{p-1})$ holds.
\end{proposition}
\begin{proof}
If $E[p^n]$ is tame, then the result is immediate since $u_{L/K}\leq1$.
Hence we may assume that $E[p^n]$ is wildly ramified. Since $p>3$,
the ramification index of the minimal extension $K'/K$ over which $E$
acquires good reduction is prime to $p$, so $K'/K$ is tamely ramified.
Replacing $K'$ by its Galois closure over $K$, if necessary, we may
assume that $K'/K$ is Galois; this replacement does not change
$e_{K'/K}$ because the extension is tame.

Let $E'/K'$ be an elliptic curve with good reduction obtained from
$E_{K'}$ by a change of Weierstrass coordinates, and put
$L'=LK'=K'(E[p^n])$. Then $L'/K$ and $L'/K'$ are Galois extensions,
and $E[p^n]\simeq E'[p^n]$ as $G_{K'}$-modules. By
Theorem~\ref{thm:BoundElliptic}, $G_{K'}^{(m)}$ acts trivially on
$E'[p^n]$ for $m>e_{K'}(n+\frac{1}{p-1})$. Hence
$u_{L'/K'}\leq e_{K'}(n+\frac{1}{p-1})$.

Choose $\alpha\in\mcal{O}_{L'}$ such that
$\mcal{O}_{L'}=\mcal{O}_K[\alpha]$; then also
$\mcal{O}_{L'}=\mcal{O}_{K'}[\alpha]$. Let $W$ denote the common wild
inertia subgroup of $\mrm{Gal}(L'/K)$ and $\mrm{Gal}(L'/K')$. These
subgroups agree because $K'/K$ is tame. Since both $L'/K$ and
$L'/K'$ are wildly ramified, the maximal lower breaks are attained on
$W$, and therefore
\begin{align*}
e_{K'/K}i_{L'/K}
&=\sup_{1\neq\sigma\in\mrm{Gal}(L'/K)}
  v_{K'}(\sigma\alpha-\alpha)
 =\sup_{1\neq\sigma\in W}v_{K'}(\sigma\alpha-\alpha)\\
&=\sup_{1\neq\sigma\in\mrm{Gal}(L'/K')}
  v_{K'}(\sigma\alpha-\alpha)
 =i_{L'/K'}.
\end{align*}
On the other hand, the transitivity of the different and the tameness
of $K'/K$ give
\[
v_{K'}(\mfrak{D}_{L'/K})
=v_{K'}(\mfrak{D}_{L'/K'})+v_{K'}(\mfrak{D}_{K'/K})
=v_{K'}(\mfrak{D}_{L'/K'})+(e_{K'/K}-1).
\]
Consequently, Corollary~\ref{cor:GaloisClosure} yields
\begin{align*}
e_{K'/K}u_{L'/K}
&=i_{L'/K'}+v_{K'}(\mfrak{D}_{L'/K'})+(e_{K'/K}-1)\\
&=u_{L'/K'}+(e_{K'/K}-1)\\
&\leq e_{K'}\left(n+\frac{1}{p-1}\right)+(e_{K'/K}-1).
\end{align*}
Dividing by $e_{K'/K}$ and using $e_{K'}=e_Ke_{K'/K}$, we obtain
\[
u_{L'/K}\leq e_K\left(n+\frac{1}{p-1}\right)
+1-\frac{1}{e_{K'/K}}.
\]
Finally, Proposition~\ref{prop:composite} gives
$u_{L'/K}=\max\{u_{L/K},u_{K'/K}\}=u_{L/K}$, because $L/K$ is wild
and $K'/K$ is tame. This proves the assertion.
\end{proof}
In Proposition \ref{prop:RamificationBoundPotGood}, 
the authors do not know whether $u_{L/K} \leq 1+e_K(n+\frac{1}{p-1})$ 
remains valid for $p=2,3$. 
However, we can prove it for $n=1$ as follows:
\begin{proposition}\label{prop:DifferentElliptic}
Let $E$ be an elliptic curve over $K$. 
Put $L=K(E[p])$. 
Then we have 
\[ u_{L/K} \leq 1 + e_K \left( 1+ \frac{1}{p-1} \right).\]
\end{proposition}
\begin{proof}
With a suitable basis of $E[p]$, we consider the representation
$G_K \to GL_{\mbb{F}_p}(E[p]) \simeq GL_2(\mbb{F}_p)$. 
The order $m=\sharp \mrm{Gal}(L/K)$ divides $\sharp GL_2(\mbb{F}_p)=p(p-1)(p^2-1)$. 
By Proposition \ref{prop:BoundRamificationBreak}, 
we have 
\[ u_{L/K} \leq 1+e_K \left( v_p(m)+\frac{1}{p-1} \right) \leq 1+e_K \left( 
1+\frac{1}{p-1} \right).\]
\end{proof}

\section{Proofs of main results}\label{Sect:Proof}
\subsection{Good reduction case}
In this section, we give proofs of the theorems in Section \ref{main:local}.
If $E$ is an elliptic curve defined over a $p$-adic field   with good ordinary reduction
and $\hat{E}$ is the formal group associated with $E$, 
we often identify the $p$-torsion part $\hat{E}[p]$ 
of $\hat{E}$ with the kernel of the reduction map $E[p]\to \bar{E}[p]$.

We start with two lemmas before moving on to the proof of the main result. 

\begin{lemma}
\label{lem:tameness}
Let $p$ be an odd prime. 
Let $\mbb{F}$ be the finite field of order $p^f$
and $F$ the unramified extension of $\mbb{Q}_p$ of degree $f$.
Let $E_{/F}$ be an elliptic curve with good ordinary reduction
and denote by $\bar{E}$ its reduction.
Assume that $\bar{E}(\mbb{F})[p]\not=0$.
Then, the following are equivalent:
\begin{itemize}
\item[{\rm (i)}] $E(F)[p]\not=0$.
\item[{\rm (ii)}] $E[p]\simeq \mbb{F}_p(1)\oplus \mbb{F}_p$ as $G_F$-modules.
\item[{\rm (iii)}] $E[p]$ is tame.
\end{itemize}
\end{lemma}

\begin{proof}
If $E(F)[p]$ is not zero, 
then $\hat{E}[p]\, (\simeq \mbb{F}_p(1))$ and $E(F)[p] \, (\simeq \mbb{F}_p)$
are non-isomorphic $G_F$-subrepresentations of $E[p]$,
which implies $E[p]=\hat{E}[p]\oplus E(F)[p]\simeq \mbb{F}_p(1)\oplus \mbb{F}_p$.
This shows (i)$\Rightarrow$ (ii).
The implications (ii)$\Rightarrow$ (i) and (ii)$\Rightarrow$ (iii) are clear.
It suffices to show  (iii)$\Rightarrow$ (ii). 
If $\bar{E}(\mbb{F})[p]\not=0$
and $E[p]$ is tame,  
the restriction to $I_F$ of the natural exact sequence 
\[
0\to \hat{E}[p]\to E[p]\to \bar{E}[p]\to 0
\] of 
$\mbb{F}_p[G_F]$-modules defines a trivial class in $\mrm{Ext}^1_{\mbb{F}_p[I_F]}(\mbb{F}_p,\mbb{F}_p(1))$.
On the other hand, the inflation-restriction exact sequence shows that 
the natural map $H^1(G_F,\mbb{F}_p(1))\to H^1(I_F,\mbb{F}_p(1))$ is injective,
which is equivalent to saying that the natural map 
\[
\mrm{Ext}^1_{\mbb{F}_p[G_F]}(\mbb{F}_p,\mbb{F}_p(1))\to \mrm{Ext}^1_{\mbb{F}_p[I_F]}(\mbb{F}_p,\mbb{F}_p(1))
\] 
is injective. Thus  the exact sequence above splits, which implies (ii).
\end{proof}

\begin{lemma}
\label{lem:red:def}
Let $E_{/\mbb{Q}_p}$ be an elliptic curve with good ordinary reduction. 
Let $\alpha$ be the non-unit root of $T^2-a_p(E)T+p=0$
and denote by $\chi_{\alpha}\colon G_{\mbb{Q}_p}\to \mbb{Z}_p^{\times}$
 the Lubin-Tate character associated with $\alpha$.
\begin{itemize}
\item[{\rm (1)}] The $G_{\mbb{Q}_p}$-action on the $p$-adic Tate module $V_p(\hat{E})$ of $\hat{E}$
is given by $\chi_{\alpha}$.
\item[{\rm (2)}] For any non-zero $P\in E[p]$ contained in the kernel of the reduction map $E[p]\to \bar{E}[p]$, 
we have 
\[
\mbb{Q}_p(P)=\mbb{Q}_p\left(\sqrt[p-1]{-a_p(E)^{-1}p}\right).
\]
\end{itemize}
\end{lemma}

\begin{proof}
(1) Let $\chi\colon G_{\mbb{Q}_p}\to \mbb{Z}_p^{\times}$ be the character obtained by the 
$G_{\mbb{Q}_p}$-action on $V_p(\hat{E})$
and $\phi\colon G_{\mbb{Q}_p}\to \mbb{Z}_p^{\times}$ be the character obtained by the 
$G_{\mbb{Q}_p}$-action on $V_p(\bar{E})$.
Put $T^2-a_p(E)T+p=f_E(T)$.
For any crystalline $\mbb{Q}_p$-representation $V$ of $G_{\mbb{Q}_p}$,
let $D_{\mrm{cris}}(V)=(B_{\mrm{cris}}\otimes_{\mbb{Q}_p} V)^{G_{\mbb{Q}_p}}$ 
be Fontaine's filtered $\vphi$-module\footnote{For the basic notion of $p$-adic Hodge theory, it is helpful for the reader to refer to \cite{Fon94a} and \cite{Fon94b}.}.
By $p$-adic Hodge theory, it is known that $f_E(T)$ coincides with the characteristic polynomial of 
the $\vphi$-module  $D_{\mrm{cris}}(V_p(E)^{\vee})$, that is, 
\[
f_E(T)=\mrm{det}(T-\vphi\mid D_{\mrm{cris}}(V_p(E)^{\vee})).
\] 
Here, $\vee$ stands for the dual representation.
Moreover, this characteristic polynomial is the product of the characteristic polynomials of 
$D_{\mrm{cris}}(\mbb{Q}_p(\chi^{-1}))$
and $D_{\mrm{cris}}(\mbb{Q}_p(\phi^{-1}))$.
Since the restriction of $\chi$ to the inertia group $I_{\mbb{Q}_p}$ coincides with the $p$-adic cyclotomic character,
for any choice of a uniformizer $\pi$ of $\mbb{Q}_p$, 
it follows from \cite[Proposition B.4]{Con11} that 
\[
\mrm{det}(T-\vphi\mid D_{\mrm{cris}}(\mbb{Q}_p(\chi^{-1})))
=T-\chi(\pi)\cdot \pi,
\]
which is independent of the choice of $\pi$
(here, we regard $\chi$ as a character of $\mbb{Q}_p^{\times}$
via the local reciprocity map $\mbb{Q}_p^{\times}\to 
\mrm{Gal}(\mbb{Q}_p^{\mrm{ab}}/\mbb{Q}_p)$).
Since $\chi(\pi)\cdot \pi$ has a positive $p$-adic valuation, we have 
$\chi(\pi)\cdot \pi=\alpha$.
By choosing  $\alpha$ as $\pi$, 
we have $\chi(\alpha)=1$. 
Since we have $\chi=\chi_{\alpha}$ on $I_{\mbb{Q}_p}$,
we find $\chi=\chi_{\alpha}$.

(2) If $p=2$, then $P$ is defined over $\mbb{Q}_2$
and thus the assertion follows.
Assume $p\ge 3$. Let $\beta$ be the unit root of $T^2-a_p(E)T+p=0$.
Let $\mcal{F}$ be  the Lubin-Tate formal group over $\mbb{Z}_p$ associated  with $\alpha$. 
By (1),  we have 
\[
\mbb{Q}_p(P)=\mbb{Q}_p(\hat{E}[p])=\mbb{Q}_p(\mcal{F}[\alpha])
=\mbb{Q}_p(x; x^p+\alpha x=0)
=\mbb{Q}_p\left(\sqrt[p-1]{-\beta^{-1}p}\right).
\]
Note that we have 
\[
u_{\mbb{Q}_p(\hat{E}[p])/\mbb{Q}_p}\le u_{\mbb{Q}_p(E[p])/\mbb{Q}_p}\le 1+\frac{1}{p-1}<2
\]
by Corollary \ref{cor:BreakBound} and $p\ge 3$.
We also note that $\beta^{-1}p\equiv a_p(E)^{-1}p\! \mod p^2$.
Therefore, the result follows
from Proposition \ref{prop:criteria}.
\end{proof}

Now we are ready to prove Theorems \ref{ord:tame}, \ref{ord:wild} and \ref{ss}.


\vspace{5mm}
\noindent
{\bf Good ordinary reduction -- Tame case:}\quad   
First we  give a proof of Theorem \ref{ord:tame}.
Thus we consider the case where $p\ge 3$, 
$E_{/\mbb{Q}_p}$ has good ordinary reduction 
and the representation $E[p]$ is tame.

\begin{proof}[Proof of Theorem \ref{ord:tame}]
(1) is an immediate consequence of Lemma \ref{lem:tameness}. We show (2).
Let $\ve\colon G_{\mbb{Q}_p}\to \mbb{F}_p^{\times}$ be the mod $p$ cyclotomic character,
$\chi\colon G_{\mbb{Q}_p}\to \mbb{F}_p^{\times}$ 
the character obtained by the $G_{\mbb{Q}_p}$-action on $\hat{E}[p]$
and $\psi\colon G_{\mbb{Q}_p}\to \mbb{F}_p^{\times}$ 
the character obtained by the $G_{\mbb{Q}_p}$-action on $\bar{E}[p]$.
The character $\psi$ is the unramified character of order $f$. 
Since $\hat{E}[p]\, (\simeq  \mbb{F}_p(\chi)=\mbb{F}_p(\ve \psi^{-1}))$ 
and $E(F)[p]\, (\simeq \mbb{F}_p(\psi))$ are non-isomorphic 
$\mbb{F}_p[G_{\mbb{Q}_p}]$-submodules of $E[p]$, 
we have 
\[
E[p]= \hat{E}[p]\oplus E(F)[p]
\] as 
$\mbb{F}_p[G_{\mbb{Q}_p}]$-modules.
Let $P$ be a non-zero point of $E[p]$.
\begin{itemize}
\item[(i)] Consider the case $P\in E(F)[p]$. For $\sigma\in G_{\mbb{Q}_p}$, 
we have $\sigma P=P$ 
$\Leftrightarrow$ $\psi(\sigma)=1$ 
$\Leftrightarrow$ $\sigma \in G_F$.
This shows $\mbb{Q}_p(P)=F$.

\item[(ii)] Consider the case $P\in \hat{E}[p]$;
this case is treated in Lemma \ref{lem:red:def}.

\item[(iii)] Consider the other cases. In this case, we have $P=Q+R$
for some non-zero $Q\in \hat{E}[p]$ and some non-zero $R\in E(F)[p]$.
For $\sigma\in G_{\mbb{Q}_p}$, 
we have $\sigma P=P$ 
$\Leftrightarrow$ $\chi(\sigma)=\psi(\sigma)=1$ 
$\Leftrightarrow$ $\ve(\sigma)=\psi(\sigma)=1$ 
$\Leftrightarrow$ $\sigma \in G_{F(\mu_p)}$.
This shows $\mbb{Q}_p(P)=F(\mu_p)$.
\end{itemize}
\end{proof}


\vspace{5mm}
\noindent
{\bf Good ordinary reduction -- Wild case:}\quad   
Next we give a proof of Theorem \ref{ord:wild}.
Thus we consider the case where $p\ge 3$, $E_{/\mbb{Q}_p}$ has good ordinary reduction 
and the representation $E[p]$ is wild.

\begin{proof}[Proof of Theorem \ref{ord:wild}]
Take a basis of $E[p]$ given by a non-zero element of $\hat{E}[p]$ 
and a lift of a non-zero point of $\bar{E}[p]$,
and 
let $\rho\colon G_{\mbb{Q}_p}\to GL_{\mbb{F}_p}(E[p])\simeq GL_2(\mbb{F}_p)$
be the representation defined by the $G_{\mbb{Q}_p}$-action on $E[p]$
with respect to the chosen basis of $E[p]$.
Since $E[p]$ is wild by Lemma \ref{lem:tameness},
it follows from \cite[Proposition 13, Corollaire]{Ser72} that
$\rho(I_{\mbb{Q}_p})=\rho(I_F)
=\begin{pmatrix}
\mbb{F}_p^{\times} & \mbb{F}_p \\
0 & 1
\end{pmatrix}$.
Since $G_F$ acts trivially on $\bar{E}[p]$,
$\rho(G_F)$ is contained in 
$\begin{pmatrix}
\mbb{F}_p^{\times} & \mbb{F}_p \\
0 & 1
\end{pmatrix}$
 and thus 
we see $\rho(G_F)=\rho(I_F)$. Hence we obtain that the extension 
$\mbb{Q}_p(E[p])/\mbb{Q}_p$
is of  degree $fp(p-1)$, and its ramification index and residue degree are  
$p(p-1)$ and $f$, respectively. Moreover, if we denote by $C_f$ the 
subgroup of $\mbb{F}_p^{\times}$ of order $f$, 
we find the following.
\[
\displaystyle \xymatrix{
\mrm{Gal}(\mbb{Q}_p(E[p])/\mbb{Q}_p(\mu_p))
\ar@{}[r]|*{\subset}
\ar[d]_{\wr} \ar^{\rho}[d]
&
\mrm{Gal}(\mbb{Q}_p(E[p])/\mbb{Q}_p)
\ar@{}[r]|*{\supset}
\ar[d]_{\wr} \ar^{\rho}[d]
&
\mrm{Gal}(\mbb{Q}_p(E[p])/F)
\ar[d]_{\wr} \ar^{\rho}[d]
\\
G_f 
\ar@{}[r]|*{\subset} 
& 
G 
\ar@{}[r]|*{\supset} 
& 
H
}
\]
Here, 
\[
G_f= \left\{ 
\begin{pmatrix}
x & \ast \\
0 & x^{-1} 
\end{pmatrix}
\mid x\in C_f 
\right\},\  
G=\begin{pmatrix}
\mbb{F}_p^{\times} & \mbb{F}_p \\
0 & C_f
\end{pmatrix},\  
H=\begin{pmatrix}
\mbb{F}_p^{\times} & \mbb{F}_p \\
0 & 1
\end{pmatrix}.
\]
In the rest of the proof,
we identify $\mrm{Gal}(\mbb{Q}_p(E[p])/\mbb{Q}_p)=G$ via $\rho$.
Under this identification, we have 
 $\mrm{Gal}(\mbb{Q}_p(E[p])/\mbb{Q}_p(\mu_p))=G_f$ and $\mrm{Gal}(\mbb{Q}_p(E[p])/F)=H$. 
Now we need the following elementary lemma.
\begin{lemma}
\label{group}
All the subgroups of $G$ of order $f(p-1)$
are of the form
\[
H_c=\left\{ 
\begin{pmatrix}
x & c(x-y) \\
0 & y 
\end{pmatrix}
\mid x\in \mbb{F}_p^{\times},\ y\in C_f  
\right\}
\]
where $c\in \mbb{F}_p$.
These subgroups are mutually conjugate. 
\end{lemma}
\begin{proof}
It suffices to show that any subgroup of $G$ of order $f(p-1)$
is of the form $H_c$ for some $c$.
Let $H$ be a subgroup of $G$ of order $f(p-1)$.
Let $\vphi\colon G\to \mbb{F}_p^{\times}\times C_f$ be the map defined by 
$\begin{pmatrix}
x & \ast \\
0 & y 
\end{pmatrix}
\mapsto (x,y)$.
Since $f(p-1)$ is prime to the order $p$ of the kernel of $\varphi$,
it follows that $\vphi$ restricted to $H$ is injective. 
Thus $\vphi$ induces $H \simeq \mbb{F}_p^{\times}\times C_f$.
In particular, $H$ is abelian. Now we suppose $H\not=H_0$. 
Then there exists an element $\sigma_0\in H$ of the form 
$\sigma_0=
\begin{pmatrix}
x_0 & u_0 \\
0 & y_0 
\end{pmatrix}
$
with  $u_0\not=0$.
Note that $\sigma_0^{f(p-1)}$ must be the identity matrix 
since $H$ is of order $f(p-1)$.
If $x_0=y_0$, then it holds
\[
\begin{pmatrix}
1 & 0 \\
0 & 1 
\end{pmatrix}
= \sigma_0^{f(p-1)}
=x_0^{f(p-1)}
\begin{pmatrix}
1 & x_0^{-1}u_0 \\
0 & 1 
\end{pmatrix}^{f(p-1)}
=x_0^{f(p-1)}
\begin{pmatrix}
1 & x_0^{-1}f(p-1)u_0 \\
0 & 1 
\end{pmatrix}
\]
but this contradicts $u_0\not=0$.
Hence we obtain $x_0\not=y_0$.
Take any element 
$\sigma=
\begin{pmatrix}
x & u \\
0 & y 
\end{pmatrix}
\in H$.
Since $H$ is abelian, the equation $\sigma \sigma_0=\sigma_0\sigma$
gives 
$xu_0+y_0u=x_0u+yu_0$. Thus we obtain $u=c(x-y)$ where 
$c=u_0(x_0-y_0)^{-1}$. Therefore, we have $H=H_c$.
Finally, for $g_c=\begin{pmatrix}1&-c\\0&1\end{pmatrix}$, we have 
$H_c=g_cH_0g_c^{-1}$. Hence the subgroups $H_c$ are mutually conjugate.
\end{proof}
Now we return to the proof of Theorem \ref{ord:wild}.
By the lemma above,
we obtain the fact that 
the field extension 
$\mbb{Q}_p(E[p])/\mbb{Q}_p$ has exactly $p$ subextensions
of degree $p$, all of which are conjugate over $\mbb{Q}_p$.
Let $K$ be one such subextension. Proposition \ref{prop:BreakDegreeP} gives 
an inequality $u_{K/\mbb{Q}_p} \geq 1+\frac{1}{p-1}$ 
for the maximal ramification breaks. 
In addition,  
Corollary \ref{cor:GaloisClosure} and \ref{cor:BreakBound} show the inequalities 
\[
u_{K/\mbb{Q}_p}\le u_{\mbb{Q}_p(E[p])/\mbb{Q}_p}\le 1+\frac{1}{p-1}.
\]
Thus, we have $u_{K/\mbb{Q}_p}=1+\frac{1}{p-1}$. 
According to Proposition \ref{prop:BreakDegreeP}, $K$ must be 
of type $\langle \lambda,m,\omega \rangle$ with $\lambda=m=1$. 
By Corollary \ref{cor:ClassificationLambda}, we find that 
$K$ is isomorphic to 
\[
\mbb{Q}_p[X]/(X^p+apX+p)
\]
for some $a\in \{1,2,\dots ,p-1\}$.  
Note that the constant $a$ is uniquely determined by 
Corollary \ref{cor:ClassificationLambda}. 
Furthermore, it follows from Proposition \ref{prop:GaloisCriteria} that 
the Galois closure $\widetilde{K}$ of $K/\mbb{Q}_p$
is 
\[
\widetilde{K}=K(\sqrt[p-1]{-ap}).
\]
Since $\widetilde{K}$ is totally ramified of degree $p(p-1)$, 
we have $\mbb{Q}_p(E[p])=F\widetilde{K}=FK(\sqrt[p-1]{-ap})$.
On the other hand, $G$ has a normal subgroup
$
\begin{pmatrix}
1 & \mbb{F}_p \\
0 & 1 
\end{pmatrix}
$
of order $p$. By the Sylow's theorem,
this is the unique subgroup of $G$ of order $p$.
Hence $\mbb{Q}_p(E[p])/\mbb{Q}_p$ has exactly one subextension of degree $f(p-1)$.
Combining this with $\mbb{Q}_p(E[p])=FK(\sqrt[p-1]{-ap})$, we have 
$F(\mu_p)=F(\sqrt[p-1]{-ap})$. In particular, we have
 $\mbb{Q}_p(E[p])=FK(\mu_p)$.

Therefore, for the proof of (1) and (2),
it is enough to show $a\equiv a_p(E)^{-2} \ \mrm{mod}\ p$. 
Now we regard $\rho$ as a homomorphism from $G=\mathrm{Gal}(\mbb{Q}_p(E[p])/\mbb{Q}_p)$ 
to $GL_2(\mbb{F}_p)$.
For a certain choice of the basis of $E[p]$, we know that 
$\rho$ 
is of the form  
\[
\rho=\begin{pmatrix}
\ve \chi^{-1}  & u \\ 
0 & \chi
\end{pmatrix}
\]
where $\ve$ is the mod $p$ cyclotomic character, 
$\chi$ is the unramified character obtained by the $G$-action on $\bar{E}[p]$ 
and $u\colon G\to \mbb{F}_p$ is a map. 
Take any $\sigma_0\in G_f$ 
with the property that it induces the $p$-th power Frobenius map 
on the residue field of $\mbb{Q}_p(E[p])$.
Since $\chi(\sigma_0)=a_p(E) \ \mrm{mod}\ p$,
we have 
$\rho(\sigma_0)=\begin{pmatrix}
a_p(E)^{-1}  & u(\sigma_0) \\ 
0 & a_p(E)
\end{pmatrix}$.
Let us denote by $W$ the wild inertia subgroup of $G$.
Note that 
$W=\mrm{Gal}(\mbb{Q}_p(E[p])/F(\mu_p))=\begin{pmatrix}
1  & \mbb{F}_p \\ 
0 & 1
\end{pmatrix}$ 
via $\rho$.
One sees immediately that $u\colon W\to \mbb{F}_p$
is an isomorphism and 
\begin{equation}
\label{map:u}
u(\sigma_0 \tau \sigma_0^{-1})=a_p(E)^{-2}u(\tau)
\end{equation}
for any $\tau\in W$.
Below, from the viewpoint of ramification theory, 
let us  construct an isomorphism $u'\colon W\to \mbb{F}_p$
satisfying 
\begin{equation}
\label{map:u'}
u'(\sigma_0 \tau \sigma_0^{-1})=a u'(\tau)
\end{equation}
for any $\tau\in W$.
Since both $u$ and $u'$ are isomorphisms from $W$ to $\mbb{F}_p$,
there exists $\lambda\in \mbb{F}_p^{\times}$ such that $u'=\lambda u$.
Comparing this identity with \eqref{map:u} and \eqref{map:u'}, 
we see that the existence of such a map $u'$ implies 
$a\equiv a_p(E)^{-2} \ \mrm{mod}\ p$.
Hence it suffices to construct $u'$ as above.

Let us explain how to construct $u'$ as above.
Take a root $\pi$ of $X^p+apX+p=0$. We may suppose $K=\mbb{Q}_p(\pi)$. 
Recall that the Galois closure $\widetilde{K}$ of $K/\mbb{Q}_p$ is $K(\gamma)$
with $\gamma$ satisfying $\gamma^{p-1}=-ap$. For any non-trivial element $\tau\in W$,
we put 
\[
u'_{\tau}=\frac{\tau(\pi)-\pi}{\gamma}.
\]
We denote by $\mathcal{O}$ the ring of integers of $\mbb{Q}_p(E[p])$ 
and $\mbf{m}$ the maximal ideal of $\mathcal{O}$.
We claim that $u'_{\tau}$ is a unit of $\mathcal{O}$
and 
\[
\bar{u}'_{\tau}:=u'_{\tau}\ \mrm{mod}\ \mbf{m}
\] 
is contained in $\mbb{F}^{\times}_p$.
Put $\delta_{\tau}=\tau(\pi)-\pi$.
Since $F(\mu_p)\cap K=\mbb{Q}_p$, 
it holds that $\delta_{\tau}\not=0$.
Since $\pi$ and $\pi+\delta_{\tau}$ are roots of $X^p+apX+p=0$,
we have 
\begin{equation}
\label{delta}
\delta_{\tau}^{p-1}+\sum^{p-1}_{i=1}\binom{p}{i}\pi^{p-i}\delta_{\tau}^{i-1}+ap=0.
\end{equation}
We see $v_p(\delta_{\tau}^{p-1})=v_p(ap)=1$.
Thus we have $v_p(u'_{\tau})=v_p(\delta_{\tau})-v_p(\gamma)=0$, which shows 
$u'_{\tau}\in \mathcal{O}^{\times}$.
Since we have 
$v_p(\delta_{\tau}^{p-1}-\gamma^{p-1})>1$ by \eqref{delta}, 
we find $\left(\frac{\delta_{\tau}}{\gamma}\right)^{p-1}\equiv 1\ \mrm{mod}\ \mbf{m}$.
Thus we have $\bar{u}'_{\tau}\in \mbb{F}^{\times}_p$.
Hence the claim follows.

By the claim, we obtain the map
\[
u'\colon W\to \mbb{F}_p
\]
defined by $u'(\tau)=\bar{u}'_{\tau}$. Note that 
$u'(1)=0$ and $u'(\tau)\not=0$ for any non-trivial element $\tau\in W$.
To check that this $u'$ satisfies the desired property,
it is enough to show that $u'$ is a homomorphism and also it satisfies \eqref{map:u'}.
Take any $\tau_1,\tau_2\in W$. Since $W$ fixes $\gamma$, we have 
\[
u'_{\tau_1\tau_2}=u'_{\tau_1}+\tau_1(u'_{\tau_2}).
\]
Since $W$ is contained in the inertia subgroup of $G$,
the above equation gives $\bar{u}'_{\tau_1\tau_2}=\bar{u}'_{\tau_1}+\bar{u}'_{\tau_2}$.
We compute $u'(\sigma_0\tau\sigma_0^{-1})$. 
Take $\eta$ such that $\eta^{p-1}=-p$
and set $\alpha:=\gamma \eta^{-1}$.
Note that we have $\alpha^{p-1}=a$ and $\mbb{Q}_p(\mu_p)=\mbb{Q}_p(\eta)$.
We denote by $\mbf{m}^{>1/(p-1)}$ the ideal of $\mathcal{O}$
consisting of all elements $x$ of $\mathcal{O}$ such that $v_p(x)>\frac{1}{p-1}$.
Since $\sigma_0(\eta)=\eta$, we have  
\begin{equation}
\label{sigma0}
\sigma_0(\gamma)=\sigma_0(\alpha)\eta\equiv \alpha^p\eta\equiv a\gamma \ \mrm{mod}\ \mbf{m}^{>1/(p-1)}
\end{equation}
Put $\pi_0=\sigma_0^{-1}(\pi)$.
This is a root of $X^p+apX+p=0$.
Since $W$ acts transitively on the set of roots of this polynomial,
we have $\pi_0=\tau_0(\pi)$ for some $\tau_0\in W$.
Using the fact that $W$ is abelian, we obtain
\begin{align*}
\sigma_0\tau\sigma_0^{-1}(\pi)-\pi
=\sigma_0(\tau(\pi_0)-\pi_0)
=\sigma_0(\tau\tau_0(\pi)-\tau_0(\pi))
=\sigma_0\tau_0(\tau(\pi)-\pi).
\end{align*}
If we denote by $[x]\in \mbb{Z}_p^{\times}$ the Teichm\"uller lift of 
$x\in \mbb{F}_p^{\times}$, we have 
\begin{align*}
\sigma_0\tau\sigma_0^{-1}(\pi)-\pi
& \equiv \sigma_0\tau_0([\bar{u}'_{\tau}]\gamma)\ \mrm{mod}\ \mbf{m}^{>1/(p-1)}\\
& \equiv \sigma_0([\bar{u}'_{\tau}])\sigma_0(\gamma)\ \mrm{mod}\ \mbf{m}^{>1/(p-1)}\\
& \equiv a[\bar{u}'_{\tau}]\gamma\ \mrm{mod}\ \mbf{m}^{>1/(p-1)}
\end{align*}
Here we used the facts that $\tau_0(\bar{u}'_{\tau})=\bar{u}'_{\tau}$ 
and $\tau_0(\gamma)=\gamma$ for the second congruence,
and also used \eqref{sigma0} for the last congruence.
Dividing the above equation by $\gamma$,
we obtain \eqref{map:u'} as desired. Consequently, 
$a\equiv a_p(E)^{-2}\pmod p$.
Let $f_a(X)=X^p+apX+p$ and $f_E(X)=X^p+a_p(E)^{-2}pX+p$. 
Proposition~\ref{prop:ultrametric} gives 
$v_p(f_a,f_E)\geq 2+\frac{1}{p}>\frac{p}{p-1}=u_{K/\mbb{Q}_p}$.
Proposition~\ref{prop:criteria} therefore shows that 
$K\simeq\mbb{Q}_p[X]/(f_E)$.
This finishes the proofs of assertions (1) and (2) of the theorem.

\medskip
It remains to prove assertions (3) and (4).
Let $P$ be a non-zero point of $E[p]$ with $P\notin \hat{E}[p]$. 
We claim that 
\[
\mbb{Q}_p(P)=FK
\]
for some 
subextension $K$ of $\mbb{Q}_p(E[p])/\mbb{Q}_p$ of degree $p$.
(Note that such $K$ must be as in the assertion (1).)
Take a non-zero point $Q$ of $\hat{E}[p]$. 
Then $\{Q,P\}$ forms a basis of $E[p]$.
We may suppose that 
$\rho
\colon G_{\mbb{Q}_p}\to GL_2(\mbb{F}_p)$ 
is defined with respect to this basis. If we write  
$\chi\colon G_{\mbb{Q}_p}\to \mbb{F}_p^{\times}$ 
the character obtained by the $G_{\mbb{Q}_p}$-action on $\bar{E}[p]$,  
we have 
$\rho=\begin{pmatrix}
\ve \chi^{-1}  & u \\ 
0 & \chi
\end{pmatrix}$
for some map $u\colon G_{\mbb{Q}_p}\to \mbb{F}_p$.
Then, $\sigma\in G$  is an element of 
$\mrm{Gal}(\mbb{Q}_p(E[p])/\mbb{Q}_p(P))$
if and only if $\chi(\sigma)=1$ and $u(\sigma)=0$.
Thus 
\[
\mrm{Gal}(\mbb{Q}_p(E[p])/\mbb{Q}_p(P))
=\begin{pmatrix}
\mbb{F}_p^{\times}  & 0 \\ 
0 & C_f
\end{pmatrix}\cap 
\begin{pmatrix}
\mbb{F}_p^{\times}  & \mbb{F}_p \\ 
0 & 1
\end{pmatrix}.
\]
Taking a subextension $K$ of $\mbb{Q}_p(E[p])/\mbb{Q}_p$ 
so that $\mrm{Gal}(\mbb{Q}_p(E[p])/K)=\begin{pmatrix}
\mbb{F}_p^{\times}  & 0 \\ 
0 & C_f
\end{pmatrix}$,
the claim follows.
For each $\sigma\in \Gamma_K$, let $S_{\sigma}$ 
be the set obtained by removing zero from the submodule of $E[p]$
generated by $\sigma P$.
We also set $\hat{S}:=\hat{E}[p]\smallsetminus \{0\}$.
It is straightforward to verify the following two properties:
\begin{itemize}
\item[(a)] $\mbb{Q}_p(Q)=F\cdot \sigma K$ for each $Q\in S_{\sigma}$
and $\mbb{Q}_p(Q)=\mbb{Q}_p(\sqrt[p-1]{-a_p(E)^{-1}p})$ 
for each $Q\in \hat{S}$ (see Lemma \ref{lem:red:def}).
\item[(b)] $F\cdot \sigma K\not=F\cdot \tau K$ for $\sigma\not=\tau$ 
in $\Gamma_K$.
\end{itemize}
By (a) and (b), the intersection $S_{\sigma}\cap S_{\tau}$ is empty for $\sigma\not=\tau$ 
in $\Gamma_K$, and similarly, 
$S_{\sigma}\cap \hat{S}$ is empty for $\sigma\in \Gamma_K$.
Since the order of $\hat{S}$ and each $S_{\sigma}$ is $p-1$, 
we find the following decomposition of $E[p]$ by disjoint unions: 
\[
E[p]=\{0\}\sqcup \hat{S} \sqcup \left( \sqcup_{\sigma\in \Gamma_K} S_{\sigma} \right).
\]
Hence we find that $E(FK)[p]$ is generated by $P$. 
In fact, if there exists a non-zero point $Q$ of $E(FK)[p]$ 
which is not generated by $P$, then $Q$ must be contained in $\hat{S}$ or  $S_{\sigma}$ for some $\sigma\not=1$ but 
this contradicts the facts that 
$FK\cap \mbb{Q}_p(\hat{E}[p])=\mbb{Q}_p\subset F$, 
$FK\cap F\cdot \sigma K=F$ for $\sigma\not=1$ and $E(F)[p]=0$.
Therefore, $P$ generates $E(FK)[p]$, and the reduction map induces
an isomorphism
\[
E(FK)[p]\xrightarrow{\sim}\bar{E}[p].
\]
Now assertions (3) and (4) follow from the arguments above.
\end{proof}


\vspace{5mm}
\noindent
{\bf Good ordinary reduction -- $p=2$:}\quad   
\begin{proof}[Proof of Theorem \ref{ord:p=2}]
Put $L=\mbb{Q}_2(E[2])$.
Since $E$ has good ordinary  reduction, 
the Galois group $\mrm{Gal}(L/\mbb{Q}_2)$
is isomorphic to a subgroup conjugate to
$\begin{pmatrix}
1 & \mbb{F}_2 \\
0 & 1
\end{pmatrix}$.
In particular, the degree of $L$ over $\mbb{Q}_2$ is at most $2$.
On the other hand, we have $u_{L/\mbb{Q}_2}\le 2$ by 
Corollary~\ref{cor:BreakBound}. If $L$ is ramified over $\mbb{Q}_2$, 
then $u_{L/\mbb{Q}_2}=2$ by the Hasse--Arf theorem. 
There are precisely two such extensions $L_1$ and $L_2$ in 
Table~\ref{table:degree2}.

Suppose that $[L:\mbb{Q}_2]=2$. The non-trivial element of 
$\mrm{Gal}(L/\mbb{Q}_2)$ acts faithfully on the three non-zero points 
of $E[2]$, and hence acts as a transposition. It therefore fixes exactly 
one non-zero point and interchanges the other two. Together with the 
origin, exactly two elements of $E[2]$ are defined over $\mbb{Q}_2$, 
whereas the other two generate $L$. This proves the theorem.
\end{proof}

\vspace{5mm}
\noindent
{\bf Good supersingular reduction:}\quad   
Finally we show Theorem \ref{ss}.
Thus we consider the case where $E_{/\mbb{Q}_p}$ has good supersingular reduction.

\begin{proof}[Proof of Theorem \ref{ss}] 
(1) It is known that, over  $\mathcal{O}_F$,  
the formal group $\hat{E}$ is isomorphic to the Lubin-Tate formal group $\mcal{F}_{/\mcal{O}_F}$ 
associated with the uniformizer $-p$. 
Kobayashi showed this fact in  \cite[Corollary 8.5 and Proposition 8.6]{Kob03}
under the assumption that $p$ is odd and 
$a_p(E)=0$.
Pollack pointed out in the proof of Theorem 3.1 \cite{Pol05} 
that we can remove the assumption $a_p(E)=0$. 
Both Kobayashi and Pollack
assumed that $p$ is odd, but Honda's theory works also for $p=2$, so the fact above follows
(as was pointed out by Sprung \cite{Spr12}).
Since  the residue degree of $\mbb{Q}_p(E[p])/\mbb{Q}_p$ is 2 by 
\cite[Proposition 12]{Ser72}, 
$F$ is contained in $\mbb{Q}_p(E[p])$. 
Writing $[-p](X)\in \mcal{O}_F[\![X]\!]$ for the power series 
representing multiplication by $-p$ on $\mcal{F}_{/\mcal{O}_F}$, we obtain
\[
\mbb{Q}_p(E[p])=F(E[p])=F(x; [-p](x)=0)\simeq F[X]/(X^{p^2-1}-p).
\]

\medskip
(2) 
Let $P$ be a non-zero point of $E[p]$ and set $L:=\mbb{Q}_p(P)$.
Denote by $\widetilde{L}$ the Galois closure of $L/\mbb{Q}_p$.
Since $E(\widetilde{L})[p]$ is a non-zero $G_{\mbb{Q}_p}$-stable 
submodule of $E[p]$ and $E[p]$ is irreducible as a $G_{\mbb{Q}_p}$-representation, we have $E(\widetilde{L})[p]=E[p]$.
This shows $\widetilde{L}=\mbb{Q}_p(E[p])$. 
One verifies immediately that any subextension field of $\mbb{Q}_p(E[p])/F$ 
is of the form $F(\sqrt[r]{p})$ for some divisor $r$ of $p^2-1$.
In particular, any subextension of $\mbb{Q}_p(E[p])/F$ is a Galois extension of $\mbb{Q}_p$.
This shows that $FL$ contains $\widetilde{L}$, which implies $FL=\mbb{Q}_p(E[p])$.
Hence the extension degree of $\mbb{Q}_p(E[p])/\mbb{Q}_p(P)$ is at most $2$ for every non-zero point $P\in E[p]$.

Now we denote by $H$ the Galois group of $\mbb{Q}_p(E[p])/\mbb{Q}_p(\sqrt[p^2-1]{p})$
and let $\sigma_0$ be the generator of $H$.
We claim that there exists a non-zero point $P\in E[p]$
such that 
\[
P+\sigma_0 P\not=0.
\] 
Assume that  $P+\sigma_0 P=0$ for any non-zero point $P\in E[p]$.
We recall that, by \cite[Proposition 9 or Proposition 12 (b)]{Ser72}, 
$E[p]$ has a structure of an $\mbb{F}_{p^2}$-vector space of dimension one
with the properties that 
$G_{\mbb{Q}_p}$-action on $E[p]$ is $\mbb{F}_{p^2}$-semi-linear 
and $I_{\mbb{Q}_p}$-action on $E[p]$ is given by the fundamental character 
$\theta\colon I_{\mbb{Q}_p}\to \mbb{F}_{p^2}^{\times}$ of level $2$, whose image has order $p^2-1$.
Fix a non-zero point $Q$ of $E[p]$ and take any $\lambda\in \mbb{F}_{p^2}$. 
By assumption, we have $\sigma_0(\lambda Q)=-\lambda Q$.
Since $G_{\mbb{Q}_p}$-action on $E[p]$ is $\mbb{F}_{p^2}$-semi-linear, 
we also have $\sigma_0(\lambda Q)=\sigma_0 \lambda.\sigma_0 Q =-(\sigma_0\lambda)Q$
(here, the second equality follows from the assumption, again).
Thus we obtain $\sigma_0 \lambda=\lambda$ for any $\lambda\in \mbb{F}_{p^2}$
but this contradicts the fact that the reduction gives an isomorphism $H\simeq \mrm{Gal}(\mbb{F}_{p^2}/\mbb{F}_p)$.
Hence the claim follows. 

Choose any $P$ as in the claim and set $P_0:=P+\sigma_0 P$. Then
\[
P_0\in E\bigl(\mbb{Q}_p(\sqrt[p^2-1]{p})\bigr)[p].
\]
Since the extension degree of $\mbb{Q}_p(E[p])/\mbb{Q}_p(P_0)$ is at most $2$, we obtain
\[
\mbb{Q}_p(P_0)=\mbb{Q}_p(\sqrt[p^2-1]{p}).
\]
Note that $G_{\mbb{Q}_p}$ acts on $E[p]$ transitively since the fundamental 
character $\theta$ is surjective. 
Thus any non-zero point of $E[p]$ is of the form 
$\sigma P_0$ for some $\sigma\in G_{\mbb{Q}_p}$.
This gives the desired result.
\end{proof}

\begin{proof}[Proof of Theorem \ref{cor:defferent}]
Suppose $\bar{E}$ is ordinary and $p \geq 3$. 
If we assume $E(F)[p] \not= 0$, then $L/\mbb{Q}_p$ is tamely ramified by 
Lemma \ref{lem:tameness}. 
Hence we have $L=F(\mu_p)$ by (1) in 
Theorem \ref{ord:tame}. 
Thus it has $u_{L/\mbb{Q}_p}=1$. 
Since the ramification index of $L/\mbb{Q}_p$ is equal to $p-1$, 
we have $v_p(\mfrak{D}_{L/\mbb{Q}_p})=\frac{p-2}{p-1}$ 
by \cite[Chapter III, Section 3, Proposition 13]{Ser68}. 
Therefore, (i) of (1) in Theorem \ref{cor:defferent} holds.
If we assume $E(F)[p]=0$, then $L/\mbb{Q}_p$ is wildly ramified by 
Lemma \ref{lem:tameness}. 
Hence we have $L=FK(\mu_p)$ with $\mbb{Q}_p$-isomorphism 
$K \simeq \mbb{Q}_p[X]/(X^p+apX+p)$ for some $a \in \{1,2,\dots,p-1\}$ 
by (1) and (2) in Theorem \ref{ord:wild}. 
According to Proposition \ref{prop:breaklambda} with $\lambda=1$, 
we have $v_p(\mfrak{D}_{L/\mbb{Q}_p})=\frac{p^2-2}{p^2-p}$ and 
$u_{L/\mbb{Q}_p}=\frac{p}{p-1}$. 
Therefore, (ii) of (1) in Theorem \ref{cor:defferent} holds.

Next, we suppose $\bar{E}$ is ordinary and $p=2$. 
By Theorem \ref{ord:p=2}, $L$ is $\mbb{Q}_2$-isomorphic to one of the following:
\[
\mbox{(i) $L=\mbb{Q}_2$,\ (ii) $L=\mbb{Q}_4$,\ (iii) $L=L_1$,\ (iv) $L=L_2$.}
\]
If $L=\mbb{Q}_2$ or $L=\mbb{Q}_4$, then we have 
$v_2(\mfrak{D}_{L/\mbb{Q}_2})=u_{L/\mbb{Q}_2}=0$. 
If $L=L_1$ or $L=L_2$, then we have $v_2(\mfrak{D}_{L/\mbb{Q}_2})=1$ 
and $u_{L/\mbb{Q}_2}=2$ by Table \ref{table:degree2}. 
Therefore, (2) in Theorem \ref{cor:defferent} holds.

Finally, we suppose $\bar{E}$ is supersingular.
Let $F/\mbb{Q}_p$ be the unramified quadratic extension. Then we have
$L=F(\sqrt[p^2-1]{p})$ by (1) in Theorem \ref{ss}. 
Since the ramification index of $L/\mbb{Q}_p$ is equal to $p^2-1$, 
we have $v_p(\mfrak{D}_{L/\mbb{Q}_p})=\frac{p^2-2}{p^2-1}$ and $u_{L/\mbb{Q}_p}=1$. 
Thus (3) in Theorem \ref{cor:defferent} holds. 
\end{proof}

\subsection{Potentially multiplicative reduction case}
From now on, 
put $L=\mbb{Q}_p(E[p])$ for an elliptic curve $E$ over $\mbb{Q}_p$. 
Let $F$ be the maximal unramified subextension in $L/\mbb{Q}_p$. 
The following group-theoretic lemma produces the degree-$p$
subextensions used below. Its final assertion identifies their Galois
closures when the tame quotient acts faithfully on the wild inertia subgroup.
\begin{lemma}\label{lem:GaloisTheory}
Let $p$ be an odd prime, let $F$ be an unramified extension of
$\mbb{Q}_p$, and let $\widetilde{L}/\mbb{Q}_p$ be a Galois extension
such that $\widetilde{L}/F(\mu_p)$ is totally ramified of degree $p$.
Put $G=\mrm{Gal}(\widetilde{L}/F)$ and
$N=\mrm{Gal}(\widetilde{L}/F(\mu_p))$. Then there exists a degree-$p$
extension $K/F$ such that $\widetilde{L}=K(\mu_p)$. If, moreover,
the conjugation action of $G/N$ on $N$ is faithful, then
$\widetilde{L}$ is the Galois closure of $K/F$.
\end{lemma}
\begin{proof}
We have $\sharp G=p(p-1)$, and $N$ is the unique Sylow $p$-subgroup
of $G$. By the Schur--Zassenhaus theorem, $N$ has a complement $H$ in
$G$, so that $G=N\rtimes H$ and $\sharp H=p-1$. Put
$K=\widetilde{L}^{H}$. Then $[K:F]=p$, and since $H\cap N=1$, the
composite of the fixed fields of $H$ and $N$ is all of
$\widetilde{L}$. Thus $\widetilde{L}=KF(\mu_p)=K(\mu_p)$.

Suppose now that the conjugation action of $G/N$ on $N$ is faithful.
Let $C$ be the core of $H$ in $G$, that is,
\[
C=\bigcap_{g\in G}gHg^{-1}.
\]
Then $C$ is normal in $G$ and contained in $H$. Since $C$ and $N$ are normal
subgroups of coprime orders, $[C,N]\subseteq C\cap N=1$, so $C$
centralizes $N$. The faithfulness of the action therefore implies
$C=1$. Hence the normal closure of the fixed field $K=\widetilde{L}^H$
is $\widetilde{L}$.
\end{proof}
\begin{proof}[Proof of Theorem \ref{thm:PotentialMulti}]
By Tate uniformization, $E$ becomes isomorphic over 
$\mbb{Q}_p(\sqrt{-c_6})$ to a Tate curve 
$\mbb{G}_{\mrm{m}}/q^{\mbb{Z}}$ for some $q\in\mbb{Z}_p$. Moreover,
$L=\mbb{Q}_p(\sqrt{-c_6},\mu_p,q^{1/p})$
by \cite[Proposition 1 and equality (4)]{Kra99}. Whenever
$[L:F(\mu_p)]=p$, this is the Kummer extension obtained by adjoining
$q^{1/p}$. The conjugation action of $\mrm{Gal}(F(\mu_p)/F)$ on
$\mrm{Gal}(L/F(\mu_p))$ is then given by the mod $p$ cyclotomic
character and is faithful. Lemma~\ref{lem:GaloisTheory} therefore
provides a degree-$p$ extension $K/F$ whose Galois closure is $L$ and
for which $L=K(\mu_p)$. By
\cite[Lemma 19.8]{FrKr20}, the equality 
$\mbb{Q}_p(q^{1/p})=\mbb{Q}_p$ is equivalent to the pair of congruences
$v_p(j)\equiv 0\pmod p, \widetilde{j}^{p-1}\equiv 1\pmod{p^2}.$
Set $d=[L:\mbb{Q}_p]$.
\par\noindent\textit{Case (1).}
Assume that either condition (1)(i) or condition (1)(ii) holds. By 
\cite[Proposition 19.3]{FrKr20},
\[
 d=
 \begin{cases}
  p-1,
    & \text{if $v_p(j)\equiv0\pmod p$ and 
      $\widetilde{j}^{p-1}\equiv1\pmod{p^2}$},\\
  p(p-1), & \text{otherwise}.
 \end{cases}
\]
\par\noindent\textit{Case (1.1).}
Here $d=p-1$, and hence $L=\mbb{Q}_p(\mu_p)$. This extension is tamely 
ramified, so $u_{L/\mbb{Q}_p}=1$.
\par\noindent\textit{Case (1.2).}
Here $d=p(p-1)$. Kraus' computation gives
$v_L(\mfrak{D}_{L/\mbb{Q}_p})=2p^2-2p-1.$
Proposition~\ref{prop:break0} therefore shows that 
$L=K(\mu_p)$ for a degree-$p$ extension $K/\mbb{Q}_p$ of type 
$\langle0\rangle$, and that
$u_{L/\mbb{Q}_p}=\frac{2p-1}{p-1}.$
By Corollary~\ref{cor:ClassificationType0},
$K\simeq \mbb{Q}_p[X]/(X^p+p(1+ap))$
for some $a\in\{0,1,\ldots,p-1\}$.
\par\noindent\textit{Case (1.3).}
In this case, \cite[Proposition 19.3]{FrKr20} and 
\cite[Th\'eor\`em 1]{Kra99} give
$d=p(p-1), v_L(\mfrak{D}_{L/\mbb{Q}_p})=p^2-2.$
By Proposition~\ref{prop:breaklambda}, 
$L=K(\mu_p)$ for a degree-$p$ extension $K/\mbb{Q}_p$ of type 
$\langle\lambda,m,\omega\rangle$ with $\lambda=m=1$, and
$u_{L/\mbb{Q}_p}=\frac{p}{p-1}.$
Corollary~\ref{cor:ClassificationLambda} yields
$K\simeq \mbb{Q}_p[X]/(X^p+apX+p)$
for some $a\in\{1,2,\ldots,p-1\}$.
\par\noindent\textit{Case (2).}
Assume that either condition (2)(i) or condition (2)(ii) holds. By 
\cite[Proposition 19.3]{FrKr20},
\[
 d=
 \begin{cases}
  2(p-1),
    & \text{if $v_p(j)\equiv0\pmod p$ and 
      $\widetilde{j}^{p-1}\equiv1\pmod{p^2}$},\\
  2p(p-1), & \text{otherwise}.
 \end{cases}
\]
We first record the quadratic part of the extension. If 
$\mbb{Q}_p(\sqrt{-c_6})/\mbb{Q}_p$ is unramified, then it is 
$\mbb{Q}_{p^2}$. If it is ramified, it is distinct from the unique 
quadratic subfield of $\mbb{Q}_p(\mu_p)$ whenever $d$ has the extra factor 
$2$. The composite of these two distinct ramified quadratic extensions 
contains the unramified quadratic extension. Consequently,
$\mbb{Q}_p(\sqrt{-c_6},\mu_p)=\mbb{Q}_{p^2}(\mu_p).$
\par\noindent\textit{Case (2.1).}
Here $d=2(p-1)$, so the preceding identity gives
$L=\mbb{Q}_{p^2}(\mu_p).$
The extension is tame, and therefore $u_{L/\mbb{Q}_p}=1$.
\par\noindent\textit{Case (2.2).}
Here $d=2p(p-1)$ and, by Kraus,
$v_L(\mfrak{D}_{L/\mbb{Q}_p})=2p^2-2p-1.$
Proposition~\ref{prop:break0} gives 
$L=K(\mu_p)$ for a degree-$p$ extension $K/\mbb{Q}_{p^2}$ of type 
$\langle0\rangle$, together with
$u_{L/\mbb{Q}_p}=\frac{2p-1}{p-1}.$
By Corollary~\ref{cor:ClassificationType0},
$K\simeq \mbb{Q}_{p^2}[X]/(X^p+p(1+ap))$
for some $a\in\mfrak{R}_{\mbb{Q}_{p^2}}$.
\par\noindent\textit{Case (2.3).}
Finally,
$d=2p(p-1), v_L(\mfrak{D}_{L/\mbb{Q}_p})=p^2-2.$
Proposition~\ref{prop:breaklambda} shows that 
$L=K(\mu_p)$ for a degree-$p$ extension $K/\mbb{Q}_{p^2}$ of type 
$\langle\lambda,m,\omega\rangle$ with $\lambda=m=1$, and that
$u_{L/\mbb{Q}_p}=\frac{p}{p-1}.$
Corollary~\ref{cor:ClassificationLambda} gives
$K\simeq \mbb{Q}_{p^2}[X]/(X^p+apX+p)$
for some $a\in\mfrak{R}_{\mbb{Q}_{p^2}}\setminus\{0\}$.
\end{proof}
\begin{proof}[Proof of Theorem \ref{thm:Multiplicative:p=2}]
By \cite[Proposition 1 and equality (4)]{Kra99},
$L=\mbb{Q}_2(\sqrt{-c_6},\sqrt{q})$
for some $q\in\mbb{Z}_2$. Since $L$ is generated by two square roots, 
$[L:\mbb{Q}_2]$ divides $4$. On the other hand, the action on $E[2]$ gives
$\mrm{Gal}(L/\mbb{Q}_2)\hookrightarrow \mrm{GL}_2(\mbb{F}_2),$
and $\sharp\mrm{GL}_2(\mbb{F}_2)=6$. Hence $[L:\mbb{Q}_2]$ also divides 
$6$, and therefore
$[L:\mbb{Q}_2]\leq 2.$
If $[L:\mbb{Q}_2]=1$, then $L=\mbb{Q}_2$. Suppose that 
$[L:\mbb{Q}_2]=2$. Table~\ref{table:degree2} lists all quadratic 
extensions of $\mbb{Q}_2$: the unramified extension $\mbb{Q}_4$ and the 
six ramified extensions $L_1,\ldots,L_6$. The same table gives 
$u_{\mbb{Q}_4/\mbb{Q}_2}=0$, 
$u_{L_i/\mbb{Q}_2}=2$ for $i=1,2$, and 
$u_{L_i/\mbb{Q}_2}=3$ for $i=3,4,5,6$. This proves all three assertions.
\end{proof}
\subsection{Potentially good reduction case}
\begin{proof}[Proof of Theorem \ref{thm:PotentialGood}]
We organize the proof into reductions, tame cases, and wild cases.
\noindent\textit{Reduction of cases (5)--(7).}
By \cite[Remarque 3]{Kra99}, if $v_p(\Delta)\geq8$, there is an elliptic 
curve $E'/\mbb{Q}_p$ such that
$c_4(E')=\frac{c_4}{p^2}, c_6(E')=\frac{c_6}{p^3}, \Delta(E')=\frac{\Delta}{p^6},$
and
$\mbb{Q}_p^{\mrm{ur}}(E'[p]) =\mbb{Q}_p^{\mrm{ur}}(E[p]).$
Consequently, cases (5), (6), and (7) reduce, respectively, to cases (1), 
(2), and (3). It remains to treat cases (1)--(4).
\par\noindent\textit{Case (4): reduction to good reduction.}
By the same twisting argument as in \cite[Remarque 3]{Kra99}, there is 
an elliptic curve $E^{\dagger}/\mbb{Q}_p$ with 
$c_4(E^{\dagger})=\frac{c_4}{p^2}$, $c_6(E^{\dagger})=\frac{c_6}{p^3}$, and 
$\Delta(E^{\dagger})=\frac{\Delta}{p^6}$. The curve $E^{\dagger}$ has good reduction, 
and its base change to $T$ is isomorphic to $E_T$. Thus the height and 
canonical-lift invariant in case (4) are those defined above. Kraus 
reduces this case to the good-reduction case \cite[Th\'eor\`eme 2]{Kra99}.
\par\noindent\textit{Case (4.1).}
The curve $E^{\dagger}$ has good ordinary reduction, $E^{\dagger}[p]$ is tame, and
$v_L(\mfrak{D}_{L/F})=p-2.$
Hence $L=F(\mu_p)$ and $u_{L/\mbb{Q}_p}=1$.
\par\noindent\textit{Case (4.2).}
The curve $E^{\dagger}$ has good ordinary reduction, $E^{\dagger}[p]$ is wild, and
$v_L(\mfrak{D}_{L/F})=p^2-2.$
Proposition~\ref{prop:breaklambda} gives
$L=K(\mu_p)$ with $K\simeq F[X]/(X^p+apX+p)$ for some 
$a\in\mfrak{R}_F\setminus\{0\}$, and $u_{L/\mbb{Q}_p}=\frac{p}{p-1}$.
\par\noindent\textit{Case (4.3).}
The curve $E^{\dagger}$ has good supersingular reduction, $E^{\dagger}[p]$ is tame, and
$v_L(\mfrak{D}_{L/F})=p^2-2$, while $e(L/F)=p^2-1$.
Let
$f(X)=X^{p^2-1}+a_1X^{p^2-2}+\cdots+\pi_F$
be an Eisenstein polynomial defining $L/F$. Proposition~\ref{prop:criteria} 
gives
$L\simeq F[X]/(X^{p^2-1}+\pi_F),$
and the extension is tame, so $u_{L/\mbb{Q}_p}=1$.
\par\noindent\textit{Cases (1)--(3): common setup.}
Kraus computes $v_L(\mfrak{D}_{L/F})$ in every subcase
\cite[Th\'eor\`eme 3 and its proof]{Kra99}. When $L/F$ is wildly
ramified, its ramification index is $p(p-1)$. Kraus' description of
the inertia action also shows that
$N=\mrm{Gal}(L/F(\mu_p))$ has order $p$ and that the conjugation action
of $\mrm{Gal}(F(\mu_p)/F)$ on $N$ is faithful. Hence
Lemma~\ref{lem:GaloisTheory} gives a degree-$p$ extension $K/F$ such
that $L=K(\mu_p)$ and $L$ is the Galois closure of $K/F$.
Consequently, $L$ is of the form considered in
Proposition~\ref{prop:breaklambda} or Proposition~\ref{prop:break0}.
\noindent\textit{The tamely ramified subcases.}
In cases (1.2), (2.2), and (3.2), Kraus' computation gives
$L=F(\mu_p), v_L(\mfrak{D}_{L/\mbb{Q}_p})=p-2.$
Thus $u_{L/\mbb{Q}_p}=1$.
In cases (1.4) and (3.4),
$[L:F]=\frac{p^2-1}{3}, v_L(\mfrak{D}_{L/\mbb{Q}_p})=\frac{p^2-4}{3}.$
Proposition~\ref{prop:criteria} gives
$L\simeq F[X]/(X^{(p^2-1)/3}+\pi_F)$
for some prime element $\pi_F$ of $F$, and $u_{L/\mbb{Q}_p}=1$.
In cases (1.5), (2.4), and (3.5),
$[L:F]=p^2-1, v_L(\mfrak{D}_{L/\mbb{Q}_p})=p^2-2.$
Again by Proposition~\ref{prop:criteria},
$L\simeq F[X]/(X^{p^2-1}+\pi_F),$
and $u_{L/\mbb{Q}_p}=1$.
\noindent\textit{The wildly ramified subcases.}
For these cases, the valuations computed by Kraus are strictly smaller than 
$2p^2-2p-1$, the value associated with type $\langle0\rangle$. Hence 
Proposition~\ref{prop:breaklambda} applies. Writing the degree-$p$ subextension 
as type $\langle\lambda,1,\omega\rangle$, we have
\[
 v_L(\mfrak{D}_{L/F})=p^2+(\lambda-1)p-\lambda-1,\quad 
 u_{L/\mbb{Q}_p}=\frac{p-1+\lambda}{p-1}.
\]
Solving for $\lambda$ in each case gives the following table.
\begin{center}
\renewcommand{\arraystretch}{1.35}
\begin{tabular}{c|c|c|c}
Case & $v_L(\mfrak{D}_{L/F})$ & $\lambda$ & $u_{L/\mbb{Q}_p}$ \\ \hline
(1.1) & $\dfrac{5p^2-4p-4}{3}$ & $\dfrac{2p+1}{3}$ & $\dfrac{5p-2}{3p-3}$ \\[1mm]
(1.3) & $\dfrac{5p^2-6p-2}{3}$ & $\dfrac{2p-1}{3}$ & $\dfrac{5p-4}{3p-3}$ \\[1mm]
(2.1) & $\dfrac{3p^2-2p-3}{2}$ & $\dfrac{p+1}{2}$ & $\dfrac{3p-1}{2p-2}$ \\[1mm]
(2.3) & $\dfrac{3p^2-4p-1}{2}$ & $\dfrac{p-1}{2}$ & $\dfrac{3}{2}$ \\[1mm]
(3.1) & $\dfrac{4p^2-2p-5}{3}$ & $\dfrac{p+2}{3}$ & $\dfrac{4p-1}{3p-3}$ \\[1mm]
(3.3) & $\dfrac{4p^2-6p-1}{3}$ & $\dfrac{p-2}{3}$ & $\dfrac{4p-5}{3p-3}$
\end{tabular}
\end{center}
Finally, Corollary~\ref{cor:ClassificationLambda} gives, in every row,
\[
 K\simeq F[X]/(X^p+apX^{\lambda}+p)
\]
for some $a\in\mfrak{R}_F\setminus\{0\}$. These are precisely the defining 
polynomials stated in cases (1.1), (1.3), (2.1), (2.3), (3.1), and (3.3). 
Together with the initial reduction, this also proves the corresponding 
assertions in cases (5)--(7).
\end{proof}
\begin{proof}[Proof of Theorem \ref{thm:Additive:p=3}]
Since $F/\mbb{Q}_3$ is unramified, the ramification of
$L/\mbb{Q}_3$ is determined by that of $L/F$. We use Kraus'
computation of the different and his description of the ramification
in \cite[Th\'eor\`eme 4 and its proof]{Kra99}.

\par\noindent\textit{Case (1.1).}
Suppose that $2v_3(c_6)\leq 4+v_3(\Delta)$ and
$\widetilde{\Delta}^{\,2}\equiv1\pmod{9}$. Kraus shows that
$L=F(\mu_3)$. Since $L/F$ is a tamely ramified quadratic extension,
we have $u_{L/\mbb{Q}_3}=1$.

\par\noindent\textit{Case (1.2).}
Suppose that $2v_3(c_6)\leq 4+v_3(\Delta)$ and
$\widetilde{\Delta}^{\,2}\not\equiv1\pmod{9}$. Kraus shows that
$L=F(\mu_3,\Delta^{1/3})$, $e(L/F)=6$, and
$v_L(\mfrak{D}_{L/F})=7$.

Put $M=\mbb{Q}_3(\mu_3,\Delta^{1/3})$ and
$K=\mbb{Q}_3(\Delta^{1/3})$. Since $F/\mbb{Q}_3$ is unramified
and $L=FM$ has ramification index $6$ over $F$, the extension
$M/\mbb{Q}_3$ is totally ramified of degree $6$. In particular,
$K/\mbb{Q}_3$ has degree $3$, and $M=K(\mu_3)$ is the Galois
closure of $K/\mbb{Q}_3$. Moreover, unramified base change gives
$v_3(\mfrak{D}_{M/\mbb{Q}_3})=\frac{7}{6}$.

Comparing Propositions~\ref{prop:breaklambda} and
\ref{prop:break0}, we find that $K/\mbb{Q}_3$ is of type
$\langle\lambda,1,\omega\rangle$ with $\lambda=1$.
Proposition~\ref{prop:breaklambda} therefore gives
$u_{L/\mbb{Q}_3}=\frac{3}{2}$.

Take $\pi_{\mbb{Q}_3}=-3$ and
$\mfrak{R}_{\mbb{Q}_3}=\{0,1,2\}$. By
Corollary~\ref{cor:ClassificationLambda}, $K$ is isomorphic to one
of the two fields
$\mbb{Q}_3[X]/(X^3+3X+3)$ and
$\mbb{Q}_3[X]/(X^3+6X+3)$.

For $\omega\in\{1,2\}$, put
$K_{\omega}=\mbb{Q}_3[X]/(X^3+3\omega X+3)$.
In the notation of Proposition~\ref{prop:GaloisCriteria},
the extension $K_{\omega}/\mbb{Q}_3$ is of type
$\langle1,1,\omega\rangle$. Since
$1+1-1\not\equiv0\pmod{2}$, it is not Galois. The same proposition
shows that its Galois closure is
$K_{\omega}(\sqrt{-3\omega})$.

The Galois closure of $K/\mbb{Q}_3$ contains
$\mbb{Q}_3(\mu_3)=\mbb{Q}_3(\sqrt{-3})$. For $\omega=1$, the
quadratic subfield of the Galois closure of $K_{\omega}$ is
$\mbb{Q}_3(\sqrt{-3})$. For $\omega=2$, it is
$\mbb{Q}_3(\sqrt{-6})$, which is distinct from
$\mbb{Q}_3(\sqrt{-3})$ because $2$ is not a square in
$\mbb{Q}_3$. Hence $\omega=1$, and therefore
$K\simeq\mbb{Q}_3[X]/(X^3+3X+3)$.

\par\noindent\textit{Case (1.3).}
Suppose that $2v_3(c_6)>4+v_3(\Delta)$. Kraus shows that, after
extending the scalars to $\overline{\mbb{F}}_3$, the inertia action
on $E[3]$ is diagonal and is described by fundamental characters
of level $2$ and order $8$. Hence the inertia image has order $8$.
Since $F$ is the maximal unramified subextension of
$L/\mbb{Q}_3$, the extension $L/F$ is totally ramified of degree
$8$. It is tame, and hence $u_{L/\mbb{Q}_3}=1$.
Let $f(X)=X^8+a_7X^7+\cdots+a_1X+a_0$ be an Eisenstein polynomial 
defining $L/F$, and put $\pi_F=a_0$. Then 
$v_F(f,X^8+\pi_F)\geq 1+\frac{1}{8}>1=u_{L/F}$. 
Proposition~\ref{prop:criteria} therefore gives 
$L\simeq F[X]/(X^8+\pi_F)$.

\par\noindent\textit{Case (2.1).}
Suppose that $v_3(\Delta)\not\equiv0\pmod{3}$. Kraus shows that
$L=F(\mu_3,\Delta^{1/3})$, $e(L/F)=6$, and
$v_L(\mfrak{D}_{L/F})=11$.

Put $M=\mbb{Q}_3(\mu_3,\Delta^{1/3})$ and
$K=\mbb{Q}_3(\Delta^{1/3})$. As in case (1.2),
$M/\mbb{Q}_3$ is totally ramified of degree $6$,
$M=K(\mu_3)$ is the Galois closure of $K/\mbb{Q}_3$, and
$v_3(\mfrak{D}_{M/\mbb{Q}_3})=\frac{11}{6}$.

Comparing Propositions~\ref{prop:breaklambda} and
\ref{prop:break0}, we find that $K/\mbb{Q}_3$ is of type
$\langle0\rangle$. Proposition~\ref{prop:break0} gives
$u_{L/\mbb{Q}_3}=\frac{5}{2}$.

Taking $\mfrak{R}_{\mbb{Q}_3}=\{0,1,2\}$ in
Corollary~\ref{cor:ClassificationType0}, we conclude that $K$ is
isomorphic to one of
$\mbb{Q}_3[X]/(X^3+3)$,
$\mbb{Q}_3[X]/(X^3+12)$, and
$\mbb{Q}_3[X]/(X^3+21)$.
This proves the theorem.
\end{proof}
\begin{proof}[Proof of Theorem \ref{thm:Additive:p=2}]
Let $G=\mrm{Gal}(L/\mbb{Q}_2)$ and
$I=\mrm{Gal}(L/F)$. The natural action on $E[2]$ gives an injection
$G\hookrightarrow\mrm{GL}_2(\mbb{F}_2)\simeq S_3$. Since $F$ is
the maximal unramified subextension of $L/\mbb{Q}_2$, the group $I$
is the inertia subgroup of $G$, and $\sharp I=e(L/F)$.

We first note that $\sharp I\neq6$. Indeed, if $\sharp I=6$, then
$I\simeq S_3$. On the other hand, the wild inertia subgroup is a
normal Sylow $2$-subgroup of $I$, whereas $S_3$ has no normal
Sylow $2$-subgroup. Thus $\sharp I\in\{1,2,3\}$. We now apply
Kraus' computation in \cite[Th\'eor\`eme 5 and its proof]{Kra99}.

\par\noindent\textit{Case (1.1).}
Suppose that $v_2(\Delta)$ is odd. Kraus shows that
$L=F(\sqrt{\Delta})$ and that $L/F$ is quadratic. Since
$v_2(\Delta)$ is odd, the ramified quadratic extension
$\mbb{Q}_2(\sqrt{\Delta})/\mbb{Q}_2$ is one of
$L_3,L_4,L_5,L_6$ in Table~\ref{table:degree2}. Hence
$L=FL_i$ for some $i\in\{3,4,5,6\}$. Since $F/\mbb{Q}_2$ is
unramified, Table~\ref{table:degree2} also gives
$u_{L/\mbb{Q}_2}=3$.

\par\noindent\textit{Case (2.1).}
Suppose that $v_2(\Delta)$ is even and
$\widetilde{\Delta}\equiv-1\pmod{4}$. Then $\Delta$ is not a
square in $\mbb{Q}_2^{\mrm{ur}}$. Kraus shows that
$L=F(\sqrt{\Delta})$ and that $L/F$ is quadratic. Since
$v_2(\Delta)$ is even, we have
$\mbb{Q}_2(\sqrt{\Delta})
=\mbb{Q}_2(\sqrt{\widetilde{\Delta}})$. The condition
$\widetilde{\Delta}\equiv-1\pmod{4}$ implies that this is one of
the two ramified quadratic extensions $L_1$ and $L_2$ in
Table~\ref{table:degree2}. Therefore $L=FL_i$ for some
$i\in\{1,2\}$, and Table~\ref{table:degree2} gives
$u_{L/\mbb{Q}_2}=2$.

\par\noindent\textit{Case (2.2).}
Suppose that $v_2(\Delta)\in\{6,8,12,14\}$ and
$3v_2(c_4)\geq8+v_2(\Delta)$. 
Kraus' argument shows that $\Delta$ is a square in
$\mbb{Q}_2^{\mrm{ur}}$. It also follows from the Newton polygon of
the polynomial whose roots are the $x$-coordinates of the non-zero
$2$-torsion points of $E$ that this polynomial has no root in
$\mbb{Q}_2^{\mrm{ur}}$. Hence no non-zero point of $E[2]$ is
defined over $\mbb{Q}_2^{\mrm{ur}}$. Consequently, $L/F$ is a
totally ramified extension of degree $3$.

The extension $L/F$ is tame. Moreover, $G$ cannot be cyclic of
order $3$, since there is no tamely and totally ramified Galois
extension of degree $3$ over $\mbb{Q}_2$. Hence $G\simeq S_3$,
and $F/\mbb{Q}_2$ is the unramified quadratic extension. Let
$K/\mbb{Q}_2$ be a degree-$3$ subextension of
$L/\mbb{Q}_2$. Then $K/\mbb{Q}_2$ is totally ramified and
$L=FK$.

There is a unique $\mbb{Q}_2$-isomorphism class of totally
ramified cubic extensions of $\mbb{Q}_2$. Indeed, such an
extension is tame, and every unit of $\mbb{Z}_2^\times$ is a
cube by Hensel's lemma. This extension is represented by
$\mbb{Q}_2[X]/(X^3+2)$. It follows that
$L\simeq F[X]/(X^3+2)$. Since $L/F$ is tamely ramified and
non-trivial, we have $u_{L/\mbb{Q}_2}=1$.

\par\noindent\textit{Case (2.3).}
In all the remaining cases, Kraus' computation gives
$v_L(\mfrak{D}_{L/F})=0$. Thus $L/F$ is unramified. Since
$F$ is the maximal unramified subextension of
$L/\mbb{Q}_2$, the extension $L/F$ must be trivial. Therefore
$L=F$ and $u_{L/\mbb{Q}_2}=0$.
\end{proof}
\bibliographystyle{alpha}
\bibliography{references}

\end{document}